\documentclass[a4paper]{article}

\usepackage{amsmath}
\usepackage{amssymb}
\usepackage{amsthm}
\usepackage{titlesec}
\titleformat*{\section}{\normalsize\bfseries}
\usepackage[margin=40mm]{geometry}
\def\R{\mathbb{R}}

\def\e{{\varepsilon}}

\def\p{\partial}
\newtheorem{thm}{Theorem}[section]
\newtheorem{lem}[thm]{Lemma}
\newtheorem{cor}[thm]{Corollary}
\newtheorem{prop}[thm]{Proposition}
\newtheorem{rem}[thm]{Remark}

\makeatletter
 \@addtoreset{equation}{section}
 
 \makeatother

\begin{document}

\title{\large{\bf{Asymptotic Behavior of Solutions to the Generalized KdV-Burgers Equation with Slowly Decaying Data}}}
\author{Ikki Fukuda\thanks{Department of Mathematics, Hokkaido University, Sapporo, Hokkaido, 060-0810, Japan.}}
\date{}

\maketitle

\footnote[0]{E-mail: i.fukuda@math.sci.hokudai.ac.jp}

\begin{abstract}
We consider the asymptotic behavior of the global solutions to the initial value problem for the generalized KdV-Burgers equation. It is known that the solution to this problem converges to a self-similar solution to the Burgers equation called a nonlinear diffusion wave. In this paper, we derive the optimal asymptotic rate to the nonlinear diffusion wave when the initial data decays slowly at spatial infinity. In particular, we investigate that how the change of the decay rate of the initial value affects the asymptotic rate to the nonlinear diffusion wave.
\end{abstract}

\indent asymptotic profile, slowly decaying data. \\[.3em]

\section{Introduction}
We study the asymptotic behavior of global solutions to the following generalized KdV-Burgers equation:
\begin{align}\label{1-1}
\begin{split}
u_{t}+(f(u))_{x}+ku_{xxx}&=u_{xx}  , \ \ t>0,\ \ x\in \R, \\
u(x, 0)&=u_{0}(x) , \ \ x\in \R, 
\end{split}
\end{align}
where $f(u)=(b/2)u^{2}+(c/3)u^{3}$ and $b\neq 0$, $c, k\in\R$. The subscripts $t$ and $x$ denote the partial derivatives with respect to $t$ and $x$, respectively. In the present paper, we consider the initial value problem \eqref{1-1} under the following condition: 
\begin{equation}\label{1-2}
\exists \alpha>1, \ \ \exists C>0 \ \ s.t. \ \ |u_{0}(x)|\le C(1+|x|)^{-\alpha}, \ \ x\in \R.  
\end{equation}
The purpose of our study is to obtain an asymptotic profile of the solution $u(x, t)$ and to examine the optimality of its asymptotic rate. Especially, we investigate how the change of the decay rate of the initial value affects its asymptotic rate. 

Concerning this problem, there are many results about the asymptotic behavior of the solutions to \eqref{1-1} when the initial data decays fast. When $k=0$, \eqref{1-1} becomes the generalized Burgers equation:
\begin{align}\label{1-3}
\begin{split}
u_{t}+(f(u))_{x}&=u_{xx}  , \ \ t>0, \ \ x\in \R, \\
u(x, 0)&=u_{0}(x) , \ \ x\in \R. 
\end{split}
\end{align}
It is known that the solution of \eqref{1-3} converges to a nonlinear diffusion wave defined by
\begin{equation}\label{1-4}
\chi(x, t)\equiv \frac{1}{\sqrt{1+t}} \chi_{*} \biggl(\frac{x}{\sqrt{1+t}}\biggl), \ \ t\ge0, \ \ x\in \R, 
\end{equation}
where
\begin{equation}\label{1-5}
\chi_{*}(x)\equiv \frac{1}{b}\frac{(e^{b\delta/2}-1)e^{-x^{2}/4}}{\sqrt{\pi}+(e^{b\delta/2}-1)\int_{x/2}^{\infty}e^{-y^{2}}dy}, \ \ \delta \equiv \int_{\R}u_{0}(x)dx
\end{equation}
(cf.~\cite{Kawashima, Liu, Matsumura and Nishihara}). Note that $\chi(x, t)$ is a solution of the standard Burgers equation:
\begin{equation}\label{1-6}
\chi_{t}+\left(\frac{b}{2}\chi^{2}\right)_{x}=\chi_{xx}, \ \ \int_{\R}\chi(x, 0)dx=\delta.  
\end{equation}
Concerning the convergence rate to the nonlinear diffusion wave $\chi(x, t)$, it was shown in Matsumura and Nishihara~\cite{Matsumura and Nishihara} that if we set 
 \begin{equation}\label{1-7}
w_{0}(x)\equiv \exp\biggl(-\frac{b}{2}\int_{-\infty}^{x}u_{0}(y)dy\biggl)-\exp\biggl(-\frac{b}{2}\int_{-\infty}^{x}\chi_{*}(y)dy\biggl)
\end{equation}
and assume that $w_{0}\in H^{2}(\R) \cap L^{1}(\R)$ and $\|w_{0}\|_{H^{2}}+\|w_{0}\|_{L^{1}}+\|u_{0}\|_{L^{1}}$ is sufficiently small, then the following estimate holds: 
\begin{equation}\label{1-8}
\|u(\cdot, t)-\chi(\cdot, t) \|_{L^{\infty}} \le C(1+t)^{-1}\log(2+t), \ \ t\ge0.
\end{equation}
Moreover, if $u_{0} \in L_{1}^{1}(\R) \cap H^{1}(\R)$ and $\|u_{0}\|_{L_{1}^{1}}+\|u_{0}\|_{H^{1}}$ is sufficiently small, then the optimality of the asymptotic rate to the nonlinear diffusion wave obtained in \eqref{1-8} was examined by Kato~\cite{Kato}. Here $L_{\beta}^{1}(\R)$ is a subset of $L^{1}(\R)$ whose elements satisfy $\|u_{0}\|_{L_{\beta}^{1}} \equiv \int_{\R}(1+|x|)^{\beta}|u_{0}(x)|dx <\infty$. Indeed, by constructing the second asymptotic profile which is the leading term of $u-\chi$, Kato~\cite{Kato} showed that if $\delta\neq0$ and $c\neq0$, then the asymptotic rate to the nonlinear diffusion wave $\chi(x, t)$ given by \eqref{1-8} is optimal with respect to the time decaying order. Here we note that $w_{0}\in H^{2}(\R) \cap L^{1}(\R)$ and $u_{0}\in L^{1}_{1}(\R) \cap H^{1}(\R)$ are realized when $u_{0}\in H^{1}(\R)$ and $\alpha>2$ in \eqref{1-2} (we can check these facts by a direct calculation. For details, see \eqref{3-8} below).  

The case of $b=k=1$ and $c=0$ (the KdV-Burgers equation) is also studied by many authors (cf.~\cite{Amick, Hayashi and Naumkin, Kaikina and Paredes, Karch 1997, Karch 1999}): 
\begin{align}\label{1-9}
\begin{split}
u_{t}+uu_{x}+u_{xxx}&=u_{xx}  , \ \ t>0, \ \ x\in \R, \\
u(x, 0)&=u_{0}(x) , \ \ x\in \R.
\end{split}
\end{align}
In particular, it was shown in Kaikina and Ruiz-Paredes~\cite{Kaikina and Paredes} that if $u_{0} \in L_{1}^{1}(\R) \cap H^{s}(\R)$ with $s>-1/2$, then the solution of \eqref{1-9} tends to the nonlinear diffusion wave $\chi(x, t)$ at the rate of $t^{-1}\log t$ and this rate is optimal. 

For the general case, in~\cite{Fukuda} the author considered \eqref{1-1} for all $b, c, k \in \R$ with $b\neq0$ and obtained a result about the asymptotic behavior of the solution of \eqref{1-1}, which unifies the results due to Kato~\cite{Kato} and Kaikina and Ruiz-Paredes~\cite{Kaikina and Paredes}. More precisely, we have the following estimate: If $u_{0}\in L^{1}_{1}(\R)\cap H^{3}(\R)$ and $\|u_{0}\|_{L^{1}_{1}}+\|u_{0}\|_{H^{3}}$ is sufficiently small, then the solution to \eqref{1-1} satisfies 
\begin{equation}\label{1-10}
\|u(\cdot, t)-\chi(\cdot, t)-V(\cdot, t)\|_{L^{\infty}} \le C(1+t)^{-1}, \ \ t\ge1,
\end{equation}
where
\begin{equation}\label{1-11}
V(x, t)\equiv -\frac{d}{4\sqrt{\pi}}\biggl(\frac{b^{2}k}{8}+\frac{c}{3}\biggl)V_{*}\biggl(\frac{x}{\sqrt{1+t}}\biggl)(1+t)^{-1}\log(1+t), 
\end{equation}
\begin{align}
V_{*}(x)&\equiv(b\chi_{*}(x)-x)e^{-x^{2}/4}\eta_{*}(x)=2\frac{d}{dx}(\eta_{*}(x)e^{-x^{2}/4}),  \label{1-12}\\
\eta_{*}(x)&\equiv \exp \biggl(\frac{b}{2}\int_{-\infty}^{x}\chi_{*}(y)dy\biggl), \ \ d\equiv \int_{\R}(\eta_{*}(y))^{-1}(\chi_{*}(y))^{3}dy. \label{1-13}
\end{align}

\noindent
From \eqref{1-10}, the triangle inequality and \eqref{1-11}, if $\delta \neq0$ and $(b^{2}k)/8+c/3\neq0$, we see that the solution $u(x, t)$ to \eqref{1-1} tends to the nonlinear diffusion wave $\chi(x, t)$ at the rate of $t^{-1}\log t$. Actually, the estimate
\begin{equation}\label{1-14}
C^{-1}(1+t)^{-1}\log(1+t)\le \|u(\cdot, t)-\chi(\cdot, t)\|_{L^{\infty}} \le C(1+t)^{-1}\log(1+t)
\end{equation}
holds for sufficiently large $t$. Therefore, this asymptotic rate $t^{-1}\log t$ is optimal with respect to the time decaying order. On the other hand, if $(b^{2}k)/8+c/3=0$, then we find the asymptotic rate to the nonlinear diffusion wave is $t^{-1}$, because $V(x, t)$ vanishes identically.
\\

All of the above previous results are corresponding to the case where the decay rate of the initial data $u_{0}$ is rapid, i.e., $\alpha>2$ in \eqref{1-2}. In the present paper, we are interested in the asymptotic behavior of the solution to \eqref{1-1} in the case that the initial data decays more slowly. For the  the semilinear heat equation, Narazaki and Nishihara~\cite{Narazaki and Nishihara} studied the asymptotic behavior of the solution for the slowly decaying data in the supercritical case. Actually, they studied the following initial value problem: 
\begin{align}\label{1-15}
\begin{split}
u_{t}-u_{xx}&=|u|^{p-1}u, \ t>0, \ x\in \R, \\
u(x, 0)&=u_{0}(x), \ x\in \R, 
\end{split}
\end{align}
where the data $u_{0}\in C^{0}(\R)$ satisfies $|u_{0}(x)|\le C(1+|x|)^{-\gamma}$ for $0<\gamma \le1$. In this case, if $p>1+2/\gamma$ (supercritical case) and the data is sufficiently small, then the asymptotic profile is given by 
\begin{equation*}
\Psi(x, t)=c_{\gamma}\int_{\R}\frac{1}{\sqrt{4\pi t}}e^{-(x-y)^{2}/4t}(1+|y|)^{-\gamma}dy, 
\end{equation*}
provided that the data satisfies $\lim_{|x|\to \infty}(1+|x|)^{\gamma}u_{0}(x)=c_{\gamma}$. Furthermore, in~\cite{Narazaki and Nishihara}, the semilinear damped wave equation 
\begin{align*}
&u_{tt}-u_{xx}+u_{t} =|u|^{p-1}, \ t>0, \ x\in \R, \\
&u(x, 0)=u_{0}(x), \ u_{t}(x, 0)=u_{1}(x), \ x\in \R, 
\end{align*}
with slowly decaying data is also treated. However, there are no results about the asymptotic behavior for the generalized KdV-Burgers equation with slowly decaying data up to the author's knowledge. Because of this, we would like to study the asymptotic profile for the solution to \eqref{1-1} when $1<\alpha \le2$ in \eqref{1-2}, which corresponds to the case of $0<\gamma \le1$. Indeed, it is natural to assume that $u_{0}\in L^{1}(\R)$ which is indicated by \eqref{1-2}, because the equation \eqref{1-1} leads to the conservation law: $\int_{\R}u(x, t)dx=\int_{\R}u_{0}(x)dx$. For this reason, we restrict our attention to the case of $1<\alpha \le2$. \\

In the following, we consider the initial value problem \eqref{1-1} with condition \eqref{1-2} when $1<\alpha \le2$. Our main results are as follows:
\newpage
\begin{thm}\label{main1}
Assume the condition \eqref{1-2} with $1<\alpha \le2$. Let $u_{0}\in H^{3}(\R)$ and $\|u_{0}\|_{L^{1}}+\|u_{0}\|_{H^{3}}$ is sufficiently small. Then \eqref{1-1} has a unique global solution $u(x, t) \in C^{0}([0, \infty); H^{3})$. Moreover, the solution satisfies the following estimate: 
\begin{align}\label{1-16}
\|u(\cdot, t)-\chi(\cdot, t)\|_{L^{\infty}}\le C\begin{cases}
(1+t)^{-\alpha/2}, &t\ge0, \ 1<\alpha<2,\\
(1+t)^{-1}\log(2+t), &t\ge0, \ \alpha=2,
\end{cases}
\end{align}
where $\chi(x, t)$ is defined by \eqref{1-4}.
\end{thm}
Furthermore, we examine the optimality of the asymptotic rate given in \eqref{1-16}. Actually, we have constructed the second asymptotic profile corresponding to the decay rate of the initial data: 
\begin{thm}\label{main2}
Assume the same conditions on $u_{0}$ in Theorem \ref{main1} are valid. We set 
\begin{equation}\label{1-17}
z_{0}(x)\equiv \eta_{*}(x)^{-1}\int_{-\infty}^{x}(u_{0}(y)-\chi_{*}(y))dy.
\end{equation} 
If there exists $\lim_{x\to \pm \infty}(1+|x|)^{\alpha-1}z_{0}(x)\equiv c_{\alpha}^{\pm}$, then the solution to \eqref{1-1} satisfies
\begin{align}
&\lim_{t\to \infty}(1+t)^{\alpha/2}\|u(\cdot, t)-\chi(\cdot, t)-Z(\cdot, t)\|_{L^{\infty}}=0, \ 1<\alpha<2, \label{1-18}\\
&\lim_{t\to \infty}\frac{(1+t)}{\log(1+t)}\|u(\cdot, t)-\chi(\cdot, t)-Z(\cdot, t)-V(\cdot, t)\|_{L^{\infty}}=0, \ \alpha=2,  \label{1-19}
\end{align}
where $\chi(x, t)$ and $V(x, t)$ are defined by \eqref{1-4} and \eqref{1-11}, respectively, while $Z(x, t)$ is defined by 
\begin{equation}\label{1-20}
Z(x, t)\equiv \int_{\R}c_{\alpha}(y)\p_{x}(G(x-y, t)\eta(x, t))(1+|y|)^{-(\alpha-1)}dy, \ \ c_{\alpha}(y)\equiv\begin{cases}
c_{\alpha}^{+}, &y\ge0, \\
c_{\alpha}^{-}, &y<0,
\end{cases}
\end{equation}
\begin{equation}\label{1-21}
G(x, t)\equiv \frac{1}{\sqrt{4\pi t}}e^{-x^{2}/4t}, \ \ \eta(x, t)\equiv \eta_{*}\biggl(\frac{x}{\sqrt{1+t}}\biggl)=\exp \biggl(\frac{b}{2}\int_{-\infty}^{x}\chi(y, t)dy\biggl),
\end{equation}
with $\eta_{*}(x)$ being defined by \eqref{1-13}. Moreover, if $\delta \neq0$, there are positive constants $\nu_{0}$ and $\nu_{1}$ such that 
\begin{align}
\| Z(\cdot, t)\|_{L^{\infty}} &\ge \nu_{0}|\beta_{0}|(1+t)^{-\alpha/2}, \ \ 1<\alpha<2, \label{1-22}\\
\| Z(\cdot, t)+V(\cdot, t)\|_{L^{\infty}} &\ge \nu_{1}|\beta_{1}|(1+t)^{-1}\log(1+t), \ \ \alpha=2 \label{1-23}
\end{align}
holds for sufficiently large $t$, where 
\begin{align}\label{1-24}
\begin{split}
\beta_{0} &\equiv (c_{\alpha}^{+}-c_{\alpha}^{-})\Gamma\biggl(\frac{3-\alpha}{2}\biggl)+\frac{(c_{\alpha}^{+}+c_{\alpha}^{-})b\chi_{*}(0)}{2-\alpha}\Gamma\biggl(2-\frac{\alpha}{2}\biggl), \\
\beta_{1} &\equiv \frac{c_{\alpha}^{+}+c_{\alpha}^{-}}{2}-d\biggl(\frac{b^{2}k}{8}+\frac{c}{3}\biggl), \ \ \Gamma (s)\equiv \int_{0}^{\infty}e^{-x}x^{s-1}dx, \ \ s>0,
\end{split}
\end{align}
while $\delta$ and $d$ are defined by \eqref{1-5} and \eqref{1-13}, respectively. 
\end{thm}
From \eqref{1-18}, \eqref{1-19}, \eqref{1-22} and \eqref{1-23}, we get the following lower bound estimate of $u-\chi$:
\begin{cor}\label{main3}
Under the same assumptions in Theorem \ref{main2}. There are positive constants $\mu_{0}$ and $\mu_{1}$ such that the following estimate 
\begin{align}\label{1-25}
\|u(\cdot, t)-\chi(\cdot, t)\|_{L^{\infty}} \ge \begin{cases}
\mu_{0}|\beta_{0}|(1+t)^{-\alpha/2}, &1<\alpha<2, \\
\mu_{1}|\beta_{1}|(1+t)^{-1}\log(1+t), &\alpha=2
\end{cases}
\end{align}
holds for sufficiently large $t$.
\end{cor}
\begin{rem}\label{rem1-4}
\rm{
From \eqref{1-16} and \eqref{1-25}, if $\delta \neq0$ and $\beta_{0}\neq 0$, we see that the solution $u(x, t)$ converges to the nonlinear diffusion wave $\chi(x, t)$ at the optimal rate of $t^{-\alpha/2}$ when $1<\alpha<2$. In the case of $\alpha=2$, the optimal asymptotic rate to $\chi(x, t)$ is $t^{-1}\log t$ under the assumptions $\delta \neq0$ and $\beta_{1}\neq 0$. 
}
\end{rem}
\begin{rem}\label{rem1-5}
\rm{
In the case of $k=0$, the similar estimate to \eqref{1-16} and \eqref{1-25} were obtained by evaluating $u-\chi$ directly (cf. Kitagawa~\cite{Kitagawa}). However, the proof of the lower bound estimate of $u-\chi$ in~\cite{Kitagawa} is only valid for the concrete perturbations such as $w_{0}(x)=(1+x^{2})^{-(\alpha-1)/2}$ ($1<\alpha \le2$) or $w_{0}(x)=0$ ($\alpha=2$). For this reason, we can say that our results are generalization of~\cite{Kitagawa}.
}
\end{rem}
\begin{rem}\label{rem1-6}
\rm{
In Narazaki and Nishihara~\cite{Narazaki and Nishihara}, they assumed that there exists\\ $\lim_{|x|\to \infty}(1+|x|)^{\gamma}u_{0}(x)=c_{\gamma}$. However, it seems a bit restrictive to assume the limit of $(1+|x|)^{\gamma}u_{0}(x)$ coinside each other. In fact, by introducing the function $c_{\alpha}(y)$ the definition of asymptotic profile $Z(x, t)$ \eqref{1-20}, we are able to handle the case where the limits are different. Besides, such a modification can be applied in the case of the heat equation and the damped wave equation treated in~\cite{Narazaki and Nishihara}. 
}
\end{rem}

This paper is organized as follows. In Section 2, we mention the global existence and the $L^{p}$-decay estimates of the solutions to \eqref{1-1}. In addition, we prepare a couple of lemmas for an auxiliary problems. Next in Section 3, we give the proof of Theorem \ref{main1}. Finally, we prove Theorem \ref{main2} in last two sections (Section 4 is for $1<\alpha<2$, Section 5 is for $\alpha=2$). The main novelty of this paper is to obtain the lower bound estimate of $u-\chi$ by constructing the second asymptotic profile for the solution to \eqref{1-1}. The main idea of constructing the second asymptotic profile is based on the asymptotic formula for the heat equation with slowly decaying data obtained in Narazaki and Nishihara~\cite{Narazaki and Nishihara}.  
\vskip10pt
\par\noindent
\textbf{\bf{Notations.}} In this paper, for $1\le p \le \infty$, $L^{p}(\R)$ denotes the usual Lebesgue spaces. In the following, for $f, g \in L^{2}(\R)\cap L^{1}(\R)$, we denote the Fourier transform of $f$ and the inverse Fourier transform of $g$ as follows:
\begin{align*}
\hat{f}(\xi)&\equiv \mathcal{F}[f](\xi)=\frac{1}{\sqrt{2\pi}}\int_{\R}e^{-ix\xi}f(x)dx,\\
\check{g}(x)&\equiv \mathcal{F}^{-1}[g](x)=\frac{1}{\sqrt{2\pi}}\int_{\R}e^{ix\xi}g(\xi)d\xi.
\end{align*}
Then, for $s\ge0$, the Sobolev spaces are defined by 
\begin{equation*}
H^{s}(\R)\equiv \biggl\{f\in L^{2}(\R); \ \|f\|_{H^{s}}\equiv \biggl(\int_{\R}(1+|\xi|^{2})^{s}|\hat{f}(\xi)|^{2}d\xi\biggl)^{1/2}<\infty \biggl\}. 
\end{equation*}
Throughout this paper, $C$ denotes various positive constants which may depends on the initial data or other parameters. However, we note that $C$ does not depend on the time variable $t$.

\section{Preliminaries}
In order to prove the main theorems, we prepare a couple of lemmas. First, we prove the global existence and the decay estimates for solutions to \eqref{1-1}. To state such a result, we introduce the Green function associated with the linear part of the equation in \eqref{1-1}:
\begin{equation*}
S(x, t)\equiv \mathcal{F}^{-1}[e^{-t|\xi|^{2}+itk\xi^{3}}](x).
\end{equation*}
By a direct calculation, we can show the estimates of $S(x, t)$ and its derivatives. For the proof, see Lemma A.1 and Lemma A.2 in~\cite{Karch 1999}.
\begin{lem}\label{Green1}
Let $l$ be a non-negative integer. Then, for $p\in[2, \infty]$, we have
\begin{align}
\| \p_{x}^{l}S(\cdot, t)\|_{L^{1}}&\le Ct^{-l/2}(1+t^{-1/4}), \ \ t>0, \label{G1}\\
\| \p_{x}^{l}S(\cdot, t)\|_{L^{p}}&\le Ct^{-(1/2)(1-1/p)-l/2}, \ \ t>0. \label{G2}
\end{align}
\end{lem}
\noindent
Moreover, for the convolution $S(t)*g$, we obtain the following estimate. For the proof, see Lemma 2.2 in~\cite{Fukuda}: 
\begin{lem}\label{Green2}
Let $s$ be a positive integer. Suppose $g\in L^{1}(\R) \cap H^{s}(\R)$. Then the estimate
\begin{equation}\label{G3}
\| \p^{l}_{x}(S(t)*g)\|_{L^{2}} \le C(1+t)^{-1/4-l/2}\|g\|_{L^{1}}+Ce^{-t}\| \p^{l}_{x}g\|_{L^{2}}, \ \ t\ge0
\end{equation}
holds for any integer $0\le l \le s$.
\end{lem}
\noindent
By using Lemma \ref{Green1} and Lemma \ref{Green2}, we are able to obtain the global existence and the decay estimates for the solutions to \eqref{1-1} as follows: 
\begin{prop}\label{GE}
Let $s$ be a positive integer. Assume that $u_{0}\in L^{1}(\R)\cap H^{s}(\R)$ and $E_{s}\equiv \|u_{0}\|_{L^{1}}+\|u_{0}\|_{H^{s}}$ is sufficiently small. Then \eqref{1-1} has a unique global solution $u(x, t) \in C^{0}([0, \infty); H^{s})$. Moreover the solution satisfies the following estimates:
\begin{align}
\| \p^{l}_{x}u(\cdot, t)\|_{L^{2}}&\le CE_{s}(1+t)^{-1/4-l/2}, \ \ t\ge0, \label{2-1}\\
\| \p^{l}_{x}u(\cdot, t)\|_{L^{1}}&\le CE_{s}t^{-l/2}(1+t^{-1/4}), \ \ t>0 \label{2-2}
\end{align}
hold for any integer $0\le l \le s$. In particular, we get 
\begin{equation}\label{2-3}
\| \p^{l}_{x}u(\cdot, t)\|_{L^{\infty}}\le CE_{s}(1+t)^{-1/2-l/2}, \ \ t\ge0
\end{equation}
for any integer $0\le l \le s-1$. 
\end{prop}
\begin{proof}
We consider the following integral equation associated with the initial value problem \eqref{1-1}:
\begin{align}\label{G4}
\begin{split}
u(t)&=S(t)*u_{0}-\int_{0}^{t}S(t-\tau)*(f(u)_{x})(\tau)d\tau \\
&=S(t)*u_{0}-\int_{0}^{t}(\p_{x}S(t-\tau))*(f(u))(\tau)d\tau. 
\end{split}
\end{align}
We solve this integral equation by using the contraction mapping principle for the mapping 
\begin{equation}\label{G5}
N[u]\equiv S(t)*u_{0}-\int_{0}^{t}(\p_{x}S(t-\tau))*(f(u))(\tau)d\tau.
\end{equation}
We set $N_{0}\equiv S(t)*u_{0}$. Let us introduce the Banach space $X$ as follows:
\begin{equation}\label{G6}
X\equiv \biggl\{u\in C^{0}([0, \infty); H^{s}); \ \|u\|_{X}\equiv \sum_{l=0}^{s}\sup_{t\ge0}(1+t)^{1/4+l/2}\|\p_{x}^{l}u(\cdot, t)\|_{L^{2}}<\infty \biggl\}.
\end{equation}
From Lemma \ref{Green2}, we have
\begin{equation}\label{G7}
\exists C_{0}>0 \ \ s.t. \ \ \|N_{0}\|_{X}\le C_{0}E_{s}.
\end{equation}
We apply the contraction mapping principle to \eqref{G5} on a closed subset $Y$ of $X$ below: 
\begin{equation*}
Y\equiv \{u\in X; \ \|u\|_{X}\le2C_{0}E_{s}\}.
\end{equation*}
Then it is sufficient to show the following estimates: 
\begin{equation}\label{G8}
\|N[u]\|_{X}\le2C_{0}E_{s}, 
\end{equation}
\begin{equation}\label{G9}
\|N[u]-N[v]\|_{X}\le \frac{1}{2}\|u-v\|_{X} 
\end{equation}
for $u, v \in Y$. If we have shown \eqref{G8} and \eqref{G9}, by using Banach fixed-point theorem, we see that \eqref{G4} has a unique global solution in $Y$. 

Here and later $E_{s}$ is assumed to be small. First, from the Sobolev inequality 
\begin{equation*}
\|F\|_{L^{\infty}}\le \sqrt{2}\|F\|_{L^{2}}^{1/2}\|F'\|_{L^{2}}^{1/2}, \ \ F\in H^{1}(\R)
\end{equation*}
for $0\le l \le s-1$, we have 
\begin{equation}\label{G10}
\|\p_{x}^{l}u(\cdot, t)\|_{L^{\infty}}\le \|u\|_{X}(1+t)^{-1/2-l/2}.
\end{equation}
Before proving \eqref{G8} and \eqref{G9}, we prepare the following estimates for $0\le l \le s$, $u, v \in Y$:
\begin{align}
\|\p_{x}^{l}(f(u)-f(v))(t)\|_{L^{1}}&\le C(\|u\|_{X}+\|v\|_{X})\|u-v\|_{X}(1+t)^{-1/2-l/2}, \label{G11}\\
\|\p_{x}^{l}(f(u)-f(v))(t)\|_{L^{2}}&\le C(\|u\|_{X}+\|v\|_{X})\|u-v\|_{X}(1+t)^{-3/4-l/2}.  \label{G12}
\end{align}
We shall prove only \eqref{G11}, since we can prove \eqref{G12} in the same way. We have from \eqref{G6} and \eqref{G10}
\begin{align*}
\begin{split}
&\|\p_{x}^{l}(f(u)-f(v))(t)\|_{L^{1}}\\
&\le C\|\p_{x}^{l}(u^{2}-v^{2})(\cdot, t)\|_{L^{1}}+C\|\p_{x}^{l}(u^{3}-v^{3})(\cdot, t)\|_{L^{1}}\\
&\le C\|\p_{x}^{l}((u+v)(u-v))(\cdot, t)\|_{L^{1}}+C\|\p_{x}^{l}((u-v)(u^{2}+uv+v^{2}))(\cdot, t)\|_{L^{1}}\\
&\le C\sum_{m=0}^{l}(\|\p_{x}^{l-m}u(\cdot, t)\|_{L^{2}}+\|\p_{x}^{l-m}v(\cdot, t)\|_{L^{2}})\|\p_{x}^{m}(u-v)(\cdot, t)\|_{L^{2}} \\
&\ \ \ \ +C\|\p_{x}^{l}(u-v)(\cdot, t)\|_{L^{2}}(\|u(\cdot, t)\|_{L^{2}}\|u(\cdot, t)\|_{L^{\infty}}\\
&\ \ \ \ +\|u(\cdot, t)\|_{L^{2}}\|v(\cdot, t)\|_{L^{\infty}}+\|v(\cdot, t)\|_{L^{2}}\|v(\cdot, t)\|_{L^{\infty}})\\
&\ \ \ \ +C\sum_{m=0}^{l-1}\sum_{n=0}^{l-m}\|\p_{x}^{m}(u-v)(\cdot, t)\|_{L^{\infty}}(\|\p_{x}^{n}u(\cdot, t)\|_{L^{2}}\|\p_{x}^{l-m-n}u(\cdot, t)\|_{L^{2}}\\
&\ \ \ \ +\|\p_{x}^{n}u(\cdot, t)\|_{L^{2}}\|\p_{x}^{l-m-n}v(\cdot, t)\|_{L^{2}}+\|\p_{x}^{n}v(\cdot, t)\|_{L^{2}}\|\p_{x}^{l-m-n}v(\cdot, t)\|_{L^{2}})\\
&\le C\sum_{m=0}^{l}(\|u\|_{X}+\|v\|_{X})(1+t)^{-1/4-(l-m)/2}\|u-v\|_{X}(1+t)^{-1/4-m/2}\\
&\ \ \ \ +C\|u-v\|_{X}(1+t)^{-1/4-l/2}(\|u\|_{X}^{2}+\|u\|_{X}\|v\|_{X}+\|v\|_{X}^{2})(1+t)^{-1/4}(1+t)^{-1/2}\\
&\ \ \ \ +C\sum_{m=0}^{l-1}\sum_{n=0}^{l-m}\|u-v\|_{X}(1+t)^{-1/2-m/2}\\
&\ \ \ \ \times(\|u\|_{X}^{2}+\|u\|_{X}\|v\|_{X}+\|v\|_{X}^{2})(1+t)^{-1/4-n/2}(1+t)^{-1/4-(l-m-n)/2}\\
&\le C(\|u\|_{X}+\|v\|_{X})\|u-v\|_{X}(1+t)^{-1/2-l/2}.
\end{split}
\end{align*}

Now we prove \eqref{G8} and \eqref{G9}. Using \eqref{G5}, we obtain 
\begin{equation}\label{G13}
(N[u]-N[v])(t)=-\int_{0}^{t}(\p_{x}S(t-\tau))*(f(u)-f(v))(\tau)d\tau \equiv I(x, t). 
\end{equation}
By Plancherel's Theorem, we have
\begin{align}\label{G14}
\begin{split}
\|\p_{x}^{l}I(\cdot, t)\|_{L^{2}} &\le \|(i\xi)^{l}\hat{I}(\xi, t)\|_{L^{2}(|\xi|\le1)}+\|(i\xi)^{l}\hat{I}(\xi, t)\|_{L^{2}(|\xi|\ge1)}\equiv I_{1}+I_{2}. 
\end{split}
\end{align}
Since 
\begin{equation*}
\int_{|\xi|\le1}|\xi|^{j}e^{-2(t-\tau)|\xi|^{2}}d\xi \le C(1+t-\tau)^{-j/2-1/2}, \ \ j\ge0, 
\end{equation*}
and \eqref{G11}, we have
\begin{align}\label{G15}
\begin{split}
I_{1}&\le C\int_{0}^{t}\|(i\xi)^{l+1}e^{-(t-\tau)|\xi|^{2}+ik(t-\tau)\xi^{3}}\mathcal{F}[f(u)-f(v)](\xi, \tau)\|_{L^{2}(|\xi|\le1)}d\tau\\
&\le C\int_{0}^{t/2}\sup_{|\xi|\le1}|\mathcal{F}[f(u)-f(v)](\xi, \tau)|\biggl(\int_{|\xi|\le1}|\xi|^{2(l+1)}e^{-2(t-\tau)|\xi|^{2}}d\xi \biggl)^{1/2}d\tau\\
&\ \ \ \ +C\int_{t/2}^{t}\sup_{|\xi|\le1}|(i\xi)^{l}\mathcal{F}[f(u)-f(v)](\xi, \tau)|\biggl(\int_{|\xi|\le1}|\xi|^{2}e^{-2(t-\tau)|\xi|^{2}}d\xi \biggl)^{1/2}d\tau\\
&\le C\int_{0}^{t/2}(1+t-\tau)^{-3/4-l/2}\|(f(u)-f(v))(\tau)\|_{L^{1}}d\tau\\
&\ \ \ \ +C\int_{t/2}^{t}(1+t-\tau)^{-3/4}\|\p_{x}^{l}(f(u)-f(v))(\tau)\|_{L^{1}}d\tau\\
&\le C(\|u\|_{X}+\|v\|_{X})\|u-v\|_{X}\biggl(\int_{0}^{t/2}(1+t-\tau)^{-3/4-l/2}(1+\tau)^{-1/2}d\tau\\
&\ \ \ \ +\int_{t/2}^{t}(1+t-\tau)^{-3/4}(1+\tau)^{-1/2-l/2}d\tau\biggl)\\
&\le C(\|u\|_{X}+\|v\|_{X})\|u-v\|_{X}(1+t)^{-1/4-l/2}.
\end{split}
\end{align}
For $|\xi|\ge1$, by using the Schwarz inequality, we have 
\begin{align*}
\begin{split}
&|(i\xi)^{l}\hat{I}(\xi, t)|\\
&=\biggl|(i\xi)^{l+1}\int_{0}^{t}e^{-(t-\tau)|\xi|^{2}+ik(t-\tau)\xi^{3}}\mathcal{F}[f(u)-f(v)](\xi, \tau)d\tau\biggl| \\
&\le C\int_{0}^{t}|\xi|e^{-(t-\tau)|\xi|^{2}}|(i\xi)^{l}\mathcal{F}[f(u)-f(v)](\xi, \tau)|d\tau \\
&\le C\biggl(\int_{0}^{t}|\xi|^{2}e^{-(t-\tau)|\xi|^{2}}d\tau \biggl)^{1/2}\biggl(\int_{0}^{t}e^{-(t-\tau)|\xi|^{2}}|(i\xi)^{l}\mathcal{F}[f(u)-f(v)](\xi, \tau)|^{2}d\tau \biggl)^{1/2} \\
&\le C\biggl(\int_{0}^{t}e^{-(t-\tau)|\xi|^{2}}|(i\xi)^{l}\mathcal{F}[f(u)-f(v)](\xi, \tau)|^{2}d\tau \biggl)^{1/2}.
\end{split}
\end{align*}
Therefore we have from \eqref{G12}
\begin{align}\label{G16}
\begin{split}
I_{2}&\le C\biggl(\int_{|\xi|\ge1}\int_{0}^{t}e^{-(t-\tau)|\xi|^{2}}|(i\xi)^{l}\mathcal{F}[f(u)-f(v)](\xi, \tau)|^{2}d\tau d\xi\biggl)^{1/2} \\
&\le C\biggl(\int_{0}^{t}e^{-(t-\tau)}\int_{|\xi|\ge1}|(i\xi)^{l}\mathcal{F}[f(u)-f(v)](\xi, \tau)|^{2}d\xi d\tau \biggl)^{1/2} \\
&\le C\biggl(\int_{0}^{t}e^{-(t-\tau)}\|\p_{x}^{l}(f(u)-f(v))(\tau)\|_{L^{2}}^{2}d\tau \biggl)^{1/2} \\
&\le C(\|u\|_{X}+\|v\|_{X})\|u-v\|_{X}\biggl(\int_{0}^{t}e^{-(t-\tau)}(1+\tau)^{-3/2-l}d\tau \biggl)^{1/2} \\
&\le C(\|u\|_{X}+\|v\|_{X})\|u-v\|_{X}(1+t)^{-3/4-l/2}. 
\end{split}
\end{align}
Combining \eqref{G13} through \eqref{G16}, we obtain 
\begin{equation*}
\|\p_{x}^{l}(N[u]-N[v])(t)\|_{L^{2}}\le C(\|u\|_{X}+\|v\|_{X})\|u-v\|_{X}(1+t)^{-1/4-l/2}, \ \ t\ge0
\end{equation*}
for $0\le l \le s$. Thus, there exists a positive constant $C_{1}>0$ such that 
\begin{equation*}
\|N[u]-N[v]\|_{X}\le C_{1}(\|u\|_{X}+\|v\|_{X})\|u-v\|_{X}\le 4C_{0}C_{1}E_{s}\|u-v\|_{X}, \ \ u, v \in Y.
\end{equation*}
Choosing $E_{s}$, which satisfies $4C_{0}C_{1}E_{s} \le1/2$, then we have \eqref{G9}. Moreover, taking $v=0$ in \eqref{G9}, it follows that 
\begin{equation*}
\|N[u]-N[0]\|_{X}\le C_{0}E_{s}.
\end{equation*}
Since $N[0]=N_{0}$, we obtain from \eqref{G7} that 
\begin{equation*}
\|N[u]\|_{X}\le \|N_{0}\|_{X}+\|N[u]-N[0]\|_{X} \le 2C_{0}E_{s}.
\end{equation*}
Therefore, we get \eqref{G8}. This completes the proof of the global existence and $L^{2}$-dacay estimate \eqref{2-1}. 

Finally, from Young's inequality and Lemma \ref{Green1}, \eqref{2-1}, \eqref{G6} and \eqref{G11}, we have
\begin{align*}
\begin{split}
\|\p_{x}^{l}u(\cdot, t)\|_{L^{1}}&\le\|\p_{x}^{l}(S(t)*u_{0})\|_{L^{1}}+\int_{0}^{t/2}\|(\p_{x}^{l+1}S(t-\tau))*f(u)(\tau)\|_{L^{1}}d\tau\\
&\ \ \ \ +\int_{t/2}^{t}\|(\p_{x}S(t-\tau))*(\p_{x}^{l}f(u))(\tau)\|_{L^{1}}d\tau \\
&\le \|\p_{x}^{l}S(\cdot, t)\|_{L^{1}}\|u_{0}\|_{L^{1}}+\int_{0}^{t/2}\|\p_{x}^{l+1}S(t-\tau)\|_{L^{1}}\|f(u)(\tau)\|_{L^{1}}d\tau\\
&\ \ \ \ +\int_{t/2}^{t}\|\p_{x}S(t-\tau)\|_{L^{1}}\|(\p_{x}^{l}f(u))(\tau)\|_{L^{1}}d\tau \\
&\le CE_{s}t^{-l/2}(1+t^{-1/4})\\
&\ \ \ \ +CE_{s}\int_{0}^{t/2}(t-\tau)^{-1/2-l/2}(1+(t-\tau)^{-1/4})(1+\tau)^{-1/2}d\tau \\
&\ \ \ \ +CE_{s}\int_{t/2}^{t}(t-\tau)^{-1/2}(1+(t-\tau)^{-1/4})(1+\tau)^{-1/2-l/2}d\tau \\
&\le CE_{s}\biggl(t^{-l/2}(1+t^{-1/4})+t^{-1/2-l/2}(1+t^{-1/4})t^{1/2}\\
&\ \ \ \ +(1+t)^{-1/2-l/2}(t^{1/2}+t^{1/4})\biggl)\\
&\le CE_{s}t^{-l/2}(1+t^{-1/4}), \ \ t>0. 
\end{split}
\end{align*}
Thus we get \eqref{2-2}. This completes the proof. 
\end{proof}

Next, we consider the nonlinear diffusion wave $\chi(x, t)$ defined by \eqref{1-4}, and the heat kernel $G(x, t)$ defined by \eqref{1-21}. A direct calculation yields
\begin{equation}\label{2-4}
|\chi(x, t)| \le C|\delta|(1+t)^{-1/2}e^{-x^{2}/4(1+t)}, \ t\ge0, \ x\in \R.
\end{equation}
Moreover, $\chi(x, t)$ and $G(x, t)$ satisfy the following estimates (for the proof, see e.g. Lemma 4.3 in~\cite{Kato and Ueda}).
\begin{lem}\label{lem2-2}
Let $l$ and $m$ be non-negative integers. Then, for $p\in[1, \infty]$, we have
\begin{align}
\| \p_{x}^{l}\p_{t}^{m}\chi(\cdot, t)\|_{L^{p}}&\le C|\delta|(1+t)^{-(1/2)(1-1/p)-l/2-m}, \ \ t\ge0, \label{2-5}\\
\| \p_{x}^{l}\p_{t}^{m}G(\cdot, t)\|_{L^{p}}&\le Ct^{-(1/2)(1-1/p)-l/2-m}, \ \ t>0. \label{2-6}
\end{align}
\end{lem}
Next, we treat the function $\eta(x, t)$ defined by \eqref{1-21}. For this function, it follows that 
\begin{align}
&\min \{1, e^{b\delta/2}\} \le \eta(x, t) \le \max \{1, e^{b\delta/2}\}, \label{2-7}\\ 
&\min \{1, e^{-b\delta/2}\} \le \eta(x, t)^{-1} \le \max \{1, e^{-b\delta/2}\}. \label{2-8}
\end{align}
Moreover, we have the following estimate (for the proof, see Corollary 2.3 in~\cite{Kato} or Lemma 5.4 in~\cite{Kato and Ueda}).
\begin{lem}\label{lem2-3}
Let $l$ be a positive integer and $p\in[1, \infty]$. If $|\delta| \le1$, then we have 
\begin{align}\label{2-9}
\| \p^{l}_{x}\eta(\cdot, t)\|_{L^{p}}+\| \p^{l}_{x}(\eta(\cdot, t)^{-1})\|_{L^{p}}&\le C|\delta|(1+t)^{-(1/2)(1-1/p)-l/2+1/2}, \ \ t\ge0.
\end{align}
\end{lem}

In the rest of this section, we introduce an auxiliary problem. We set $\psi=u-\chi$, where $u(x, t)$ is the solution to \eqref{1-1} and $\chi(x, t)$ is the nonlinear diffusion wave defined by \eqref{1-4}. Then, the perturbation $\psi(x, t)$ satisfies the following equation:
\begin{align*}
\psi_{t}+(b\chi \psi)_{x}-\psi_{xx}=-\biggl(\frac{b}{2}\psi^{2}\biggl)_{x}-\biggl(\frac{c}{3}u^{3}\biggl)_{x}-ku_{xxx}.
\end{align*} 
To consider the above equation, we prepare the following auxiliary problem:
\begin{align}\label{2-10}
\begin{split}
z_{t}+(b\chi z)_{x}-z_{xx}&=\p_{x}\lambda(x, t)  , \ \ t>0, \ \ x\in \R, \\
z(x, 0)&=z_{0}(x) , \ \ x\in \R, 
\end{split}
\end{align}
where $\lambda(x, t)$ is a given regular function decaying fast enough at spatial infinity. If we set 
\begin{align}\label{2-11}
\begin{split}
U[h](x, t, \tau)\equiv\int_{\R}\p_{x}(G(x-y, t-\tau)\eta(x, t))(\eta(y, \tau))^{-1}\biggl(\int_{-\infty}^{y}h(\xi)d\xi\biggl)dy&,\\
0\le \tau<t, \ \ x\in \R&, 
\end{split}
\end{align}
then we have the following formula (for the proof, see Lemma 3.3 in~\cite{Fukuda} or Lemma 3.1 in~\cite{Kato}):
\begin{lem}\label{lem2-4}
Let $z_{0}(x)$ be a sufficiently regular function decaying at spatial infinity. Then we can get the smooth solution of  \eqref{2-10} which satisfies the following formula:  
\begin{equation}\label{2-12}
z(x, t)=U[z_{0}](x, t, 0)+\int_{0}^{t}U[\p_{x}\lambda(\tau)](x, t, \tau)d\tau, \ \ t>0, \ \ x\in \R.
\end{equation}
\end{lem}
This explicit representation formula \eqref{2-12} plays an important roles in the proof of Theorem \ref{main2}. Especially, for the second term of \eqref{2-12}, we have the following estimate:
\begin{lem}\label{lem2-5}
Assume that $|\delta| \le1$. Then the following estimate 
\begin{align}\label{2-13}
\begin{split}
&\biggl\| \int_{0}^{t}U[\p_{x} \lambda(\tau)](\cdot, t, \tau)d\tau\biggl\|_{L^{\infty}} \\
&\le C\sum_{n=0}^{1}(1+t)^{-1/2+n/2}\\
&\ \ \ \ \times \biggl(\int_{0}^{t/2}(t-\tau)^{-1/2-n/2}\| \lambda(\cdot, \tau)\|_{L^{1}}d\tau+\int_{t/2}^{t}(t-\tau)^{-n/2}\| \lambda(\cdot, \tau)\|_{L^{\infty}}d\tau \biggl)
\end{split}
\end{align}
holds for $t>0$. 
\end{lem}
\begin{proof}
From \eqref{2-11}, we obtain
\begin{align*}
\begin{split}
&\int_{0}^{t}U[\p_{x} \lambda(\tau)](x, t, \tau)d\tau\\
&=\int_{0}^{t}\int_{\R}\p_{x}(G(x-y, t-\tau)\eta(x, t))(\eta(y, \tau))^{-1}\lambda(y, \tau)dyd\tau \\
&=\sum_{n=0}^{1}\p_{x}^{1-n}\eta(x, t)\biggl(\int_{0}^{t/2}+\int_{t/2}^{t}\biggl)\int_{\R}\p_{x}^{n}G(x-y, t-\tau)(\eta(y, \tau))^{-1}\lambda(y, \tau)dyd\tau.
\end{split}
\end{align*}
By using Young's inequality, Lemma \ref{lem2-2}, Lemma \ref{lem2-3} and \eqref{2-8}, we obtain 
\begin{align*}
\begin{split}
&\biggl\| \int_{0}^{t}U[\p_{x} \lambda(\tau)](\cdot, t, \tau)d\tau\biggl\|_{L^{\infty}} \\
&\le C\sum_{n=0}^{1}\|\p_{x}^{1-n}\eta(\cdot, t)\|_{L^{\infty}}\biggl(\int_{0}^{t/2}\| \p_{x}^{n}G(t-\tau)*(\eta^{-1}\lambda)(\tau)\|_{L^{\infty}}d\tau\\
&\ \ \ \ +\int_{t/2}^{t}\|\p_{x}^{n}G(t-\tau)*(\eta^{-1}\lambda)(\tau)\|_{L^{\infty}}d\tau \biggl) \\
&\le C\sum_{n=0}^{1}(1+t)^{-1/2+n/2}\\
&\ \ \ \ \times \biggl(\int_{0}^{t/2}(t-\tau)^{-1/2-n/2}\| \lambda(\cdot, \tau)\|_{L^{1}}d\tau+\int_{t/2}^{t}(t-\tau)^{-n/2}\| \lambda(\cdot, \tau)\|_{L^{\infty}}d\tau \biggl).
\end{split}
\end{align*}
\end{proof}

\section{Proof of Theorem \ref{main1}}
In this section, we shall prove Theorem \ref{main1}. First, in order to obtain the upper bound estimates of $u-\chi$, we transform the equation \eqref{1-1} and \eqref{1-6} by using the Hopf-Cole transformation used in~\cite{Cole, Hopf, Matsumura and Nishihara}. We set 
\begin{equation}\label{3-1}
\rho (x, t)\equiv \exp \biggl(-\frac{b}{2}\int_{-\infty}^{x}u(y, t)dy\biggl), 
\end{equation}
then we have the following equation: 
\begin{equation}\label{3-2}
\rho_{t}-\rho_{xx}=\frac{b}{2}\rho \biggl(\frac{c}{3}u^{3}+ku_{xx}\biggl).
\end{equation}
Also, for $\eta(x, t)$ defined by \eqref{1-21}, we have
\begin{equation}\label{3-3}
(\eta^{-1})_{t}-(\eta^{-1})_{xx}=0.
\end{equation}
Therefore, if we set 
\begin{equation}\label{3-4}
w(x, t)=\rho(x, t)-\eta(x, t)^{-1}, 
\end{equation}
then we have the following initial value problem for $w(x, t)$ from \eqref{3-2} and \eqref{3-3}: 
\begin{align}\label{3-5}
\begin{split}
w_{t}-w_{xx}&=\frac{b}{2}\rho \biggl(\frac{c}{3}u^{3}+ku_{xx}\biggl), \ \ t>0,\ \ x\in \R, \\
w(x, 0)&=w_{0}(x), \ \ x\in \R, 
\end{split}
\end{align}
where $w_{0}(x)$ is defined by \eqref{1-7}. Our first step to prove Theorem \ref{main1} is to show the following proposition:
\begin{prop}\label{prop3-1}
Assume the same conditions on $u_{0}$ in Theorem \ref{main1} are valid. Then we have
\begin{align}\label{3-6}
\|\p_{x}^{l}w(\cdot, t)\|_{L^{\infty}}\le C\begin{cases}
(1+t)^{-(\alpha-1)/2-l/2}, &t\ge1, \ 1<\alpha<2, \\
(1+t)^{-1/2-l/2}\log(2+t), &t\ge1, \ \alpha=2
\end{cases}
\end{align}
for $l=0, 1$, where $w(x, t)$ is the solution to \eqref{3-5}. 
\end{prop}
\begin{proof}
Applying the Duhamel principle to \eqref{3-5}, we have the following integral equation:
\begin{align}\label{3-7}
\begin{split}
w(t)&=G(t)*w_{0}+\frac{bc}{6}\int_{0}^{t}G(t-\tau)*(\rho u^{3})(\tau)d\tau\\
&\ \ \ \ +\frac{bk}{2}\int_{0}^{t}G(t-\tau)*(\rho u_{xx})(\tau)d\tau\\
&\equiv I_{1}+I_{2}+I_{3}.
\end{split}
\end{align} 
Before evaluating $I_{1}$, $I_{2}$ and $I_{3}$, we prove the following estimate for the initial data $w_{0}(x)$:
\begin{equation}\label{3-8}
|w_{0}(x)|\le C(1+|x|)^{-(\alpha-1)}, \ \ x\in \R. 
\end{equation}
 By the mean value theorem, we obtain 
 \begin{equation*}
 |w_{0}(x)|\le C\ \biggl|\int_{-\infty}^{x}(u_{0}(y)-\chi_{*}(y))dy\biggl|.
 \end{equation*}
 If $x<0$, from \eqref{1-2} and \eqref{1-5}, we have
\begin{align*}
|w_{0}(x)|&\le C\int_{-\infty}^{x}(|u_{0}(y)|+|\chi_{*}(y)|)dy \\
&\le C\int_{-\infty}^{x}(1+|y|)^{-\alpha}dy+C\int_{-\infty}^{x}(1+|y|)^{-N}dy \ \ (\forall N\ge0) \\
&\le C\int_{-\infty}^{x}(1-y)^{-\alpha}dy\le C(1-x)^{-(\alpha-1)}=C(1+|x|)^{-(\alpha-1)}. 
\end{align*}
On the other hand, since $\int_{\R}u_{0}(x)dx=\int_{\R}\chi_{*}(x)dx=\delta$, if $x>0$, similarly we have
\begin{align*}
|w_{0}(x)|&\le C\biggl|\int_{-\infty}^{x}(u_{0}(y)-\chi_{*}(y))dy\biggl| \\
&=C\biggl|\int_{-\infty}^{x}u_{0}(y)dy-\int_{\R}u_{0}(y)dy-\int_{-\infty}^{x}\chi_{*}(y)dy+\int_{\R}\chi_{*}(y)dy\biggl| \\
&\le C\biggl|\int_{x}^{\infty}u_{0}(y)dy\biggl|+\biggl|\int_{x}^{\infty}\chi_{*}(y)dy\biggl| \\
&\le C\int_{x}^{\infty}(1+|y|)^{-\alpha}dy+C\int_{x}^{\infty}(1+|y|)^{-N}dy \ \ (\forall N\ge0) \\
&\le C\int_{x}^{\infty}(1+y)^{-\alpha}dy\le C(1+x)^{-(\alpha-1)}=C(1+|x|)^{-(\alpha-1)}. 
\end{align*}
Thus we get \eqref{3-8}. 

Next, we prove the uniform boundedness of $\rho(x, t)$: 
\begin{equation}\label{3-9}
\|\rho(\cdot, t)\|_{L^{\infty}} \le C, \ \ t\ge0.
\end{equation}
To prove \eqref{3-9}, it suffices to show 
\begin{equation}\label{3-10}
\|w(\cdot, t)\|_{L^{\infty}} \le C, \ \ t\ge0,  
\end{equation}
since $\rho=w+\eta^{-1}$ and \eqref{2-8}. By using \eqref{3-7}, \eqref{3-4}, Young's inequality, \\
Lemma \ref{lem2-2}, \eqref{3-8}, \eqref{2-8} and Proposition \ref{GE}, we have
\begin{align*}
\begin{split}
&\|w(\cdot, t)\|_{L^{\infty}}\\
&\le \|G(\cdot, t)\|_{L^{1}}\|w_{0}\|_{L^{\infty}}\\
&\ \ \ \ +C\int_{0}^{t}\|G(\cdot,t-\tau)\|_{L^{1}}(\|w(\cdot, \tau)\|_{L^{\infty}}+\|\eta(\cdot, \tau)^{-1}\|_{L^{\infty}})\|u^{3}(\cdot, \tau)\|_{L^{\infty}}d\tau \\
&\ \ \ \ +C\int_{0}^{t}\|G(\cdot,t-\tau)\|_{L^{1}}(\|w(\cdot, \tau)\|_{L^{\infty}}+\|\eta(\cdot, \tau)^{-1}\|_{L^{\infty}})\|u_{xx}(\cdot, \tau)\|_{L^{\infty}}d\tau \\
&\le C+C'\int_{0}^{t}\|w(\cdot, \tau)\|_{L^{\infty}}(1+\tau)^{-3/2}d\tau, \ \ t\ge0. 
\end{split}
\end{align*}
Therefore, by using Gronwall's Lemma, we obtain
\begin{equation*}
\|w(\cdot, t)\|_{L^{\infty}}\le C\exp \biggl(C' \int_{0}^{t}(1+\tau)^{-3/2}d\tau\biggl)\le C, \ \ t\ge0. 
\end{equation*}

Now, we start with evaluation of $I_{1}$, $I_{2}$ and $I_{3}$. From Lemma \ref{lem2-2}, we have
\begin{align}\label{3-11}
\begin{split}
&\int_{\R} |\p_{x}^{l}G(x-y, t)|(1+|y|)^{-(\alpha-1)}dy \\
&=C\biggl(\int_{|y|\ge \sqrt{1+t}-1}+\int_{|y|\le \sqrt{1+t}-1}\biggl)|\p_{x}^{l}G(x-y, t)|(1+|y|)^{-(\alpha-1)}dy \\
&\le C\biggl(\sup_{|y|\ge \sqrt{1+t}-1}(1+|y|)^{-(\alpha-1)}\biggl)\int_{|y|\ge \sqrt{1+t}-1}|\p_{x}^{l}G(x-y, t)|dy\\
&\ \ \ \ +C\biggl(\sup_{|y|\le \sqrt{1+t}-1}|\p_{x}^{l}G(x-y, t)|\biggl)\int_{|y|\le \sqrt{1+t}-1}(1+|y|)^{-(\alpha-1)}dy\\
&\le C(1+t)^{-(\alpha-1)/2}\|\p_{x}^{l}G(\cdot, t)\|_{L^{1}}+C\|\p_{x}^{l}G(\cdot, t)\|_{L^{\infty}}\int_{|y|\le \sqrt{1+t}-1}(1+|y|)^{-(\alpha-1)}dy\\
&\le C(1+t)^{-(\alpha-1)/2-l/2}+Ct^{-1/2-l/2}\int_{0}^{\sqrt{1+t}-1}(1+y)^{-(\alpha-1)}dy \\
&\le C(1+t)^{-(\alpha-1)/2-l/2}+Ct^{-1/2-l/2}\begin{cases}
(1+t)^{-(\alpha-1)/2+1/2}, &1<\alpha<2,  \\
\log(2+t), &\alpha=2, 
\end{cases}\\
&\le C\begin{cases}
(1+t)^{-(\alpha-1)/2-l/2}, &t\ge1, \ 1<\alpha<2, \\
(1+t)^{-1/2-l/2}\log(2+t), &t\ge1, \ \alpha=2.
\end{cases}
\end{split}
\end{align} 
Therefore, we obtain from \eqref{3-8} and \eqref{3-11}
\begin{align}\label{3-12}
\|\p_{x}^{l}I_{1}(\cdot, t)\|_{L^{\infty}}\le C\begin{cases}
(1+t)^{-(\alpha-1)/2-l/2}, &t\ge1, \ 1<\alpha<2, \\
(1+t)^{-1/2-l/2}\log(2+t), &t\ge1, \ \alpha=2 
\end{cases}
\end{align}
for $l=0, 1$. By Young's inequality, \eqref{3-9}, Lemma \ref{lem2-2} and Proposition \ref{GE}, we have
\begin{align}\label{3-13}
\begin{split}
&\|\p_{x}^{l}I_{2}(\cdot, t)\|_{L^{\infty}}\\
&\le C\biggl(\int_{0}^{t/2}+\int_{t/2}^{t}\biggl)\|\p_{x}^{l}G(t-\tau)*(\rho u^{3})(\tau)\|_{L^{\infty}}d\tau\\
&\le C\int_{0}^{t/2}\|\p_{x}^{l}G(\cdot, t-\tau)\|_{L^{\infty}}\|u^{3}(\cdot, \tau)\|_{L^{1}}d\tau\\
&\ \ \ \ +C\int_{t/2}^{t}\|\p_{x}^{l}G(\cdot, t-\tau)\|_{L^{1}}\|u^{3}(\cdot, \tau)\|_{L^{\infty}}d\tau\\
&\le C\int_{0}^{t/2}\|\p_{x}^{l}G(\cdot, t-\tau)\|_{L^{\infty}}\|u(\cdot, \tau)\|_{L^{\infty}}\|u(\cdot, t)\|_{L^{2}}^{2}d\tau\\
&\ \ \ \ +C\int_{t/2}^{t}\|\p_{x}^{l}G(\cdot, t-\tau)\|_{L^{1}}\|u^{3}(\cdot, \tau)\|_{L^{\infty}}d\tau\\
&\le C\int_{0}^{t/2}(t-\tau)^{-1/2-l/2}(1+\tau)^{-1}d\tau+C\int_{t/2}^{t}(t-\tau)^{-l/2}(1+\tau)^{-3/2}d\tau\\
&\le Ct^{-1/2-l/2}\log(2+t)+C(1+t)^{-1/2-l/2}\\
&\le C(1+t)^{-1/2-l/2}\log(2+t), \ \ t\ge1
\end{split}
\end{align}
for $l=0, 1$. Finally, we evaluate $I_{3}$. Since $\rho_{x}=-(b/2)\rho u$ from \eqref{3-1}, making the integration by parts for the first part of the integral, it follows that 
\begin{align}\label{3-14}
\begin{split}
I_{3}(x, t)&=\frac{bk}{2}\biggl(\int_{0}^{t/2}+\int_{t/2}^{t}\biggl)\int_{\R}G(x-y, t-\tau)\rho(y, \tau)u_{yy}(y, \tau)dyd\tau\\
&=\frac{bk}{2}\int_{0}^{t/2}\int_{\R}\p_{x}G(x-y, t-\tau)\rho(y, \tau)u_{y}(y, \tau)dyd\tau \\
&\ \ \ \ +\frac{b^{2}k}{4}\int_{0}^{t/2}\int_{\R}G(x-y, t-\tau)\rho(y, \tau)u(y, \tau)u_{y}(y, \tau)dyd\tau \\
&\ \ \ \ +\frac{bk}{2}\int_{t/2}^{t}\int_{\R}G(x-y, t-\tau)\rho(y, \tau)u_{yy}(y, \tau)dyd\tau.
\end{split}
\end{align}
Therefore, in the same way to get \eqref{3-13}, we have from \eqref{3-14}
\begin{align}\label{3-15}
\begin{split}
&\|\p_{x}^{l}I_{3}(\cdot, t)\|_{L^{\infty}}\\
&\le C\int_{0}^{t/2}\|\p_{x}^{l+1}G(t-\tau)*(\rho u_{x})(\tau)\|_{L^{\infty}}d\tau\\
&\ \ \ +C\int_{0}^{t/2}\|\p_{x}^{l}G(t-\tau)*(\rho uu_{x})(\tau)\|_{L^{\infty}}d\tau\\
&\ \ \ +C\int_{t/2}^{t}\|\p_{x}^{l}G(t-\tau)*(\rho u_{xx})(\tau)\|_{L^{\infty}}\\
&\le C\int_{0}^{t/2}\|\p_{x}^{l+1}G(\cdot, t-\tau)\|_{L^{\infty}}\|u_{x}(\cdot, \tau)\|_{L^{1}}d\tau \\
&\ \ \ +C\int_{0}^{t/2}\|\p_{x}^{l}G(\cdot, t-\tau)\|_{L^{\infty}}\|u(\cdot, \tau)\|_{L^{2}}\|u_{x}(\cdot, \tau)\|_{L^{2}}d\tau\\
&\ \ \ +C\int_{t/2}^{t}\|\p_{x}^{l}G(\cdot, t-\tau)\|_{L^{1}}\|u_{xx}(\cdot, t)\|_{L^{\infty}}d\tau\\
&\le C\int_{0}^{t/2}(t-\tau)^{-1-l/2}\tau^{-1/2}(1+\tau^{-1/4})d\tau+C\int_{0}^{t/2}(t-\tau)^{-1/2-l/2}(1+\tau)^{-1}d\tau\\
&\ \ \ \ +C\int_{t/2}^{t}(t-\tau)^{-l/2}(1+\tau)^{-3/2}d\tau \\
&\le C(1+t)^{-1/2-l/2}+C(1+t)^{-1/2-l/2}\log(2+t)+C(1+t)^{-1/2-l/2} \\
&\le C(1+t)^{-1/2-l/2}\log(2+t), \ \ t\ge1
\end{split}
\end{align}
for $l=0, 1$. Summing up \eqref{3-7}, \eqref{3-12},  \eqref{3-13} and  \eqref{3-15}, we obtain \eqref{3-6}. 
\end{proof}

\begin{proof}[\rm{\bf{End of the proof of Theorem \ref{main1}}}]
From \eqref{3-1} and \eqref{1-21}, it follows that 
\begin{equation*}
u=-\frac{2}{b}(\log \rho)_{x}, \ \ \chi=\frac{2}{b}(\log \eta)_{x}. 
\end{equation*}
Therefore, we have
\begin{align}\label{3-16}
\begin{split}
u-\chi&=-\frac{2}{b}(\log \rho+\log \eta)_{x}=-\frac{2}{b}(\log(w+\eta^{-1})+\log \eta)\\
&=-\frac{2}{b}(\log(\eta w+1))_{x}=-\frac{2}{b}\frac{\eta_{x}w+\eta w_{x}}{\eta w+1} \\
&=-\frac{2}{b}\frac{w}{w+\eta^{-1}}\frac{\eta_{x}}{\eta}-\frac{2}{b}\frac{w_{x}}{w+\eta^{-1}}=-\frac{w\chi}{\rho}-\frac{2}{b}\frac{w_{x}}{\rho}, 
\end{split}
\end{align}
since $\eta_{x}=(b/2)\eta \chi$ and $\rho=w+\eta^{-1}$. By using \eqref{2-2}, we have 
\begin{equation*}
\biggl|\int_{-\infty}^{x}u(y, t)dy\biggl| \ \le \|u(\cdot, t)\|_{L^{1}}\le C(1+t^{-1/4})\le C=:C_{0}, \ \ t\ge1.
\end{equation*}
Thus, from \eqref{3-1}, we get the lower bound estimate of $\rho(x, t)$:
\begin{equation}\label{3-17}
\inf_{t\ge1, x\in \R}\rho(x, t)\ge e^{-|b|C_{0}/2}.
\end{equation}
Therefore, using \eqref{3-16}, Lemma \ref{lem2-2}, Proposition \ref{prop3-1} and \eqref{3-17}, if $1<\alpha<2$, we have 
\begin{align}\label{3-18}
\begin{split}
\|u(\cdot, t)-\chi(\cdot, t)\|_{L^{\infty}}&\le C\|w(\cdot, t)\|_{L^{\infty}}\|\chi(\cdot, t)\|_{L^{\infty}}+C\|\p_{x}w(\cdot, t)\|_{L^{\infty}} \\
&\le C(1+t)^{-(\alpha-1)/2}(1+t)^{-1/2}+C(1+t)^{-\alpha/2} \\
&\le C(1+t)^{-\alpha/2}, \ t\ge1.
\end{split}
\end{align}
On the other hand, if $\alpha=2$, it follows that 
\begin{align}\label{3-19}
\begin{split}
\|u(\cdot, t)-\chi(\cdot, t)\|_{L^{\infty}}&\le C\|w(\cdot, t)\|_{L^{\infty}}\|\chi(\cdot, t)\|_{L^{\infty}}+C\|\p_{x}w(\cdot, t)\|_{L^{\infty}} \\
&\le C(1+t)^{-1/2}\log(2+t)(1+t)^{-1/2}\\
&\ \ \ \ +C(1+t)^{-1}\log(2+t) \\
&\le C(1+t)^{-1}\log(2+t), \ t\ge1.
\end{split}
\end{align}
For $0\le t \le1$, from Proposition \ref{GE} and Lemma \ref{lem2-2}, we obtain 
\begin{align}\label{3-20}
\begin{split}
\|u(\cdot, t)-\chi(\cdot, t)\|_{L^{\infty}}&\le C\|u(\cdot, t)\|_{L^{\infty}}+C\|\chi(\cdot, t)\|_{L^{\infty}} \\
&\le CE_{s}(1+t)^{-1/2}+C|\delta|(1+t)^{-1/2} \\
&\le C, \ \ 0\le t \le1. 
\end{split}
\end{align}
Combining \eqref{3-18} through \eqref{3-20}, we get \eqref{1-16} for all $t\ge0$. This completes the proof. 
\end{proof}

In the rest of this section, we will show the $L^{p}$-decay estimate of $u-\chi$ for $p\in[2, \infty]$. Especially, the following $L^{2}$-decay estimate will be used in the proof of Theorem \ref{main2}. 
\begin{prop}\label{prop3-2}
Assume the same conditions on $u_{0}$ in Theorem \ref{main1} are valid. Then, the following estimate holds:
\begin{equation}\label{3-21}
\|u(\cdot, t)-\chi(\cdot, t)\|_{L^{2}}\le C\begin{cases}
(1+t)^{-\alpha/2+1/4}, &t\ge0, \ 1<\alpha<2, \\
(1+t)^{-3/4}\log(2+t), &t\ge0, \ \alpha=2.
\end{cases}
\end{equation} 
\end{prop}
\begin{proof}
It suffices to derive the following $L^{2}$-decay estimate of $\p_{x}w(x, t)$:
\begin{equation}\label{3-22}
\|\p_{x}w(\cdot, t)\|_{L^{2}}\le C\begin{cases}
(1+t)^{-\alpha/2+1/4}, &t\ge1, \ 1<\alpha<2,\\
(1+t)^{-3/4}\log(2+t), &t\ge1, \ \alpha=2.  
\end{cases}
\end{equation}
Indeed, if we have shown \eqref{3-22}, then from \eqref{3-16}, Proposition \ref{prop3-1}, Lemma \ref{lem2-2} and \eqref{3-17}, it follows that 
\begin{align}\label{3-23}
\begin{split}
\|u(\cdot, t)-\chi(\cdot, t)\|_{L^{2}}&\le C\|w(\cdot, t)\|_{L^{\infty}}\|\chi(\cdot, t)\|_{L^{2}}+C\|\p_{x}w(\cdot, t)\|_{L^{2}}\\
&\le C(1+t)^{-(\alpha-1)/2}(1+t)^{-1/4}+C(1+t)^{-\alpha/2+1/4} \\
&\le C(1+t)^{-\alpha/2+1/4}, \ t\ge1 
\end{split}
\end{align}
for $1<\alpha<2$. While, the following estimate holds in the case of $\alpha=2$: 
\begin{align}\label{3-24}
\begin{split}
\|u(\cdot, t)-\chi(\cdot, t)\|_{L^{2}}&\le C\|w(\cdot, t)\|_{L^{\infty}}\|\chi(\cdot, t)\|_{L^{2}}+C\|\p_{x}w(\cdot, t)\|_{L^{2}}\\
&\le C(1+t)^{-1/2}\log(2+t)(1+t)^{-1/4}\\
&\ \ \ \ +C(1+t)^{-3/4}\log(2+t) \\
&\le C(1+t)^{-3/4}\log(2+t), \ t\ge1.
\end{split}
\end{align}
For $0\le t \le1$, from Proposition \ref{GE}, and Lemma \ref{lem2-2}, we obtain 
\begin{align}\label{3-25}
\begin{split}
\|u(\cdot, t)-\chi(\cdot, t)\|_{L^{2}}&\le C\|u(\cdot, t)\|_{L^{2}}+C\|\chi(\cdot, t)\|_{L^{2}} \\
&\le CE_{s}(1+t)^{-1/4}+C|\delta|(1+t)^{-1/4} \\
&\le C, \ \ 0\le t \le1. 
\end{split}
\end{align}
Combining \eqref{3-23} through \eqref{3-25}, we get \eqref{3-21} for all $t\ge0$.  

To prove \eqref{3-22}, we consider the following integral equation again: 
\begin{align}\tag{\ref{3-7}}
\begin{split}
w(t)&=G(t)*w_{0}+\frac{bc}{6}\int_{0}^{t}G(t-\tau)*(\rho u^{3})(\tau)d\tau\\
&\ \ \ \ +\frac{bk}{2}\int_{0}^{t}G(t-\tau)*(\rho u_{xx})(\tau)d\tau\\
&\equiv I_{1}+I_{2}+I_{3}.
\end{split}
\end{align}
For the initial data, using \eqref{1-7}, \eqref{1-2} and \eqref{1-5}, we have
\begin{align}\label{3-26}
\begin{split}
&|w_{0}'(x)|\\
&\le C\biggl|\exp \biggl(-\frac{b}{2}\int_{-\infty}^{x}u_{0}(y)dy\biggl)\biggl||u_{0}(x)|+C\biggl|\exp \biggl(-\frac{b}{2}\int_{-\infty}^{x}\chi_{*}(y)dy\biggl)\biggl||\chi_{*}(x)|\\
&\le C \exp \biggl(\frac{|b|\|u_{0}\|_{L^{1}}}{2}\biggl)(1+|x|)^{-\alpha}+C\exp \biggl(\frac{|b|\|\chi_{*}\|_{L^{1}}}{2}\biggl)(1+|x|)^{-N} \ \ (\forall N\ge0)\\
&\le C(1+|x|)^{-\alpha}, \ \ x\in \R. 
\end{split}
\end{align}
Now we shall evaluate $I_{1}$, $I_{2}$ and $I_{3}$. Splitting the $y$-integral and making the integration by parts, we have
\begin{align}\label{3-27}
\begin{split}
\p_{x}I_{1}(x, t)&=\int_{\R}\p_{x}G(x-y, t)w_{0}(y)dy \\
&=\biggl(\int_{-\infty}^{1-\sqrt{1+t}}+\int_{\sqrt{1+t}-1}^{\infty}\biggl)\p_{x}G(x-y, t)w_{0}(y)dy\\
&\ \ \ \ +\int_{|y|\le \sqrt{1+t}-1}\p_{x}G(x-y, t)w_{0}(y)dy\\
&=-G(x-1+\sqrt{1+t}, t)w_{0}(1-\sqrt{1+t})\\
&\ \ \ \ +G(x-\sqrt{1+t}-1, t)w_{0}(\sqrt{1+t}-1)\\
&\ \ \ \ +\biggl(\int_{-\infty}^{1-\sqrt{1+t}}+\int_{\sqrt{1+t}-1}^{\infty}\biggl)G(x-y, t)w'_{0}(y)dy\\
&\ \ \ \ +\int_{|y|\le \sqrt{1+t}-1}\p_{x}G(x-y, t)w_{0}(y)dy\\
&\equiv I_{1.1}+I_{1.2}+I_{1.3}+I_{1.4}.
\end{split}
\end{align} 
Using Lemma \ref{lem2-2} and \eqref{3-8}, we have
\begin{align}\label{3-28}
\begin{split}
\|I_{1.1}(\cdot, t)\|_{L^{2}}&\le C\|G(\cdot, t)\|_{L^{2}}|w_{0}(1-\sqrt{1+t})|\\
&\le Ct^{-1/4}(1+t)^{-(\alpha-1)/2}\\
&\le C(1+t)^{-\alpha/2+1/4}, \ t\ge1.
\end{split}
\end{align}
Similarly, we obtain
\begin{equation}\label{3-29}
\|I_{1.2}(\cdot, t)\|_{L^{2}}\le C(1+t)^{-\alpha/2+1/4}, \ t\ge1.
\end{equation}
From Lemma \ref{lem2-2} and \eqref{3-26}, we obtain  
\begin{align}\label{3-30}
\begin{split}
&\|I_{1.3}(\cdot, t)\|_{L^{2}}\\
&\le C\biggl(\int_{-\infty}^{1-\sqrt{1+t}}+\int_{\sqrt{1+t}-1}^{\infty}\biggl)\|G(\cdot-y, t)\|_{L^{2}}|w'_{0}(y)|dy\\
&\le Ct^{-1/4}\int_{-\infty}^{1-\sqrt{1+t}}(1-y)^{-\alpha}dy+Ct^{-1/4}\int_{\sqrt{1+t}-1}^{\infty}(1+y)^{-\alpha}dy\\
&=Ct^{-1/4}\biggl[\frac{-1}{1-\alpha}(1-y)^{1-\alpha}\biggl]_{-\infty}^{1-\sqrt{1+t}}+Ct^{-1/4}\biggl[\frac{1}{1-\alpha}(1+y)^{1-\alpha}\biggl]_{\sqrt{1+t}-1}^{\infty}\\
&\le Ct^{-1/4}(1+t)^{1/2-\alpha/2}\\
&\le C(1+t)^{-\alpha/2+1/4}, \ t\ge1.
\end{split}
\end{align} 
On the other hand, we have from Lemma \ref{lem2-2} and \eqref{3-8} 
\begin{align}\label{3-31}
\begin{split}
\|I_{1.4}(\cdot, t)\|_{L^{2}}\le&\ \int_{|y|\le \sqrt{1+t}-1}\|\p_{x}G(\cdot-y, t)\|_{L^{2}}|w_{0}(y)|dy\\
\le&\ Ct^{-3/4}\int_{|y|\le\sqrt{1+t}-1}(1+|y|)^{-(\alpha-1)}dy\\
\le&\ Ct^{-3/4}\int_{0}^{\sqrt{1+t}-1}(1+y)^{-(\alpha-1)}dy\\
\le&\ C\begin{cases}
(1+t)^{-\alpha/2+1/4}, &t\ge1, \ 1<\alpha<2, \\
(1+t)^{-3/4}\log(2+t), &t\ge1, \ \alpha=2. 
\end{cases}
\end{split}
\end{align} 
Therefore, combining \eqref{3-27} through \eqref{3-31}, we obtain 
\begin{equation}\label{3-32}
\|\p_{x}I_{1}(\cdot, t)\|_{L^{2}}\le C\begin{cases}
(1+t)^{-\alpha/2+1/4}, &t\ge1, \ 1<\alpha<2,  \\
(1+t)^{-3/4}\log(2+t), &t\ge1, \ \alpha=2. 
\end{cases}
\end{equation}
By Young's inequality, \eqref{3-9}, Lemma \ref{lem2-2} and Proposition \ref{GE}, we have
\begin{align}\label{3-33}
\begin{split}
&\|\p_{x}I_{2}(\cdot, t)\|_{L^{2}}\\
&\le C\biggl(\int_{0}^{t/2}+\int_{t/2}^{t}\biggl)\|\p_{x}G(t-\tau)*(\rho u^{3})(\tau)\|_{L^{2}}d\tau\\
&\le C\int_{0}^{t/2}\|\p_{x}G(\cdot, t-\tau)\|_{L^{2}}\|u(\cdot, \tau)\|_{L^{\infty}}\|u(\cdot, t)\|_{L^{2}}^{2}d\tau\\
&\ \ \ \ +C\int_{t/2}^{t}\|\p_{x}G(\cdot, t-\tau)\|_{L^{1}}\|u(\cdot, t)\|_{L^{\infty}}^{2}\|u(\cdot, \tau)\|_{L^{2}}d\tau\\
&\le C\int_{0}^{t/2}(t-\tau)^{-3/4}(1+\tau)^{-1}d\tau+C\int_{t/2}^{t}(t-\tau)^{-1/2}(1+\tau)^{-5/4}d\tau\\
&\le Ct^{-3/4}\log(2+t)+C(1+t)^{-3/4}\\
&\le C(1+t)^{-3/4}\log(2+t), \ \ t\ge1.
\end{split}
\end{align}
Finally, we evaluate $I_{3}$. Using \eqref{3-14}, in the same way to get \eqref{3-15}, it follows that 
\begin{align}\label{3-34}
\begin{split}
&\|\p_{x}I_{3}(\cdot, t)\|_{L^{2}}\\
&\le C\int_{0}^{t/2}\|\p_{x}^{2}G(\cdot, t-\tau)\|_{L^{2}}\|u_{x}(\cdot, t)\|_{L^{1}}d\tau\\
&\ \ \ \ +C\int_{0}^{t/2}\|\p_{x}G(\cdot, t-\tau)\|_{L^{2}}\|u(\cdot, \tau)\|_{L^{2}}\|u_{x}(\cdot, \tau)\|_{L^{2}}d\tau\\
&\ \ \ \ +C\int_{t/2}^{t}\|\p_{x}G(\cdot, t-\tau)\|_{L^{1}}\|u_{xx}(\cdot, t)\|_{L^{2}}d\tau\\
&\le C\int_{0}^{t/2}(t-\tau)^{-5/4}\tau^{-1/2}(1+\tau^{-1/4})d\tau+C\int_{0}^{t/2}(t-\tau)^{-3/4}(1+\tau)^{-1}d\tau\\
&\ \ \ \ +C\int_{t/2}^{t}(t-\tau)^{-1/2}(1+\tau)^{-5/4}d\tau \\
&\le C(1+t)^{-3/4}+C(1+t)^{-3/4}\log(2+t)+C(1+t)^{-3/4} \\
&\le C(1+t)^{-3/4}\log(2+t), \ \ t\ge1. 
\end{split}
\end{align}
Summing up \eqref{3-7} and \eqref{3-32} through \eqref{3-34}, we obtain \eqref{3-22}. This completes the proof.   
\end{proof}

Finally, by the interpolation inequality, it follows that 
\begin{equation*}
\|u(\cdot, t)-\chi(\cdot, t)\|_{L^{p}}\le \|u(\cdot, t)-\chi(\cdot, t)\|_{L^{\infty}}^{1-2/p}\|u(\cdot, t)-\chi(\cdot, t)\|_{L^{2}}^{2/p}
\end{equation*}
for $p\in[2, \infty]$. Therefore, \eqref{1-16} and \eqref{3-21} lead to the following estimate: 
\begin{cor}\label{cor3-3}
Assume the same conditions on $u_{0}$ in Theorem \ref{main1} are valid. Then, for $p\in[2, \infty]$, we have
\begin{equation*}
\|u(\cdot, t)-\chi(\cdot, t)\|_{L^{p}}\le \begin{cases}
(1+t)^{-\alpha/2+1/2p}, &t\ge0, \ 1<\alpha<2, \\
(1+t)^{-1+1/2p}\log(2+t), &t\ge0, \ \alpha=2.
\end{cases}
\end{equation*} 
\end{cor}

\section{Proof of Theorem \ref{main2} for $1<\alpha < 2$}
In this section, we shall prove Theorem \ref{main2} in the case of $1<\alpha<2$. Namely, we prove \eqref{1-18} and \eqref{1-22}. First, we set  
\begin{equation}\label{4-1}
\psi(x, t)\equiv u(x, t)-\chi(x, t), \ \ \psi_{0}(x)\equiv u_{0}(x)-\chi_{*}(x).
\end{equation}
Then we have the following equation: 
\begin{align}\label{4-2}
\begin{split}
\psi_{t}+(b\chi \psi)_{x}-\psi_{xx}&=-\biggl(\frac{b}{2}\psi^{2}\biggl)_{x}-\biggl(\frac{c}{3}u^{3}\biggl)_{x}-ku_{xxx}, \ t>0, \ x\in \R, \\
\psi(x, 0)&=\psi_{0}(x), \ x\in \R.
\end{split}
\end{align}
From Lemma \ref{lem2-4}, we obtain 
\begin{align}\label{4-3}
\begin{split}
\psi(x, t)=&\ U[\psi_{0}](x, t, 0)-\frac{b}{2}\int_{0}^{t}U[\p_{x}\psi^{2}(\tau)](x, t, \tau)d\tau\\
&-\frac{c}{3}\int_{0}^{t}U[\p_{x}u^{3}(\tau)](x, t, \tau)d\tau-k\int_{0}^{t}U[\p_{x}^{3}u(\tau)](x, t, \tau)d\tau.
\end{split}
\end{align}
For the first term of the above equation \eqref{4-3}, we have the following asymptotic formula. This formula plays an essential role in the proof of \eqref{1-18} and \eqref{1-19}. The idea of the proof is based on the method used in Narazaki and Nishihara~\cite{Narazaki and Nishihara}.
\begin{prop}\label{prop4-1}
Assume the same conditions on $u_{0}$ in Theorem \ref{main2} are valid. Then we have 
\begin{align}
&\lim_{t\to \infty}(1+t)^{\alpha/2}\|U[\psi_{0}](\cdot, t, 0)-Z(\cdot, t)\|_{L^{\infty}}=0, \ 1<\alpha<2, \label{4-4}\\
&\lim_{t\to \infty}\frac{(1+t)}{\log(1+t)}\|U[\psi_{0}](\cdot, t, 0)-Z(\cdot, t)\|_{L^{\infty}}=0, \ \alpha=2,  \label{4-5}
\end{align}
where $Z(x, t)$ is defined by \eqref{1-20}. 
\end{prop}
\begin{proof}
From the definition of $U$ given by \eqref{2-11} we have 
\begin{align}\label{4-6}
\begin{split}
U[\psi_{0}](x, t, 0)=&\int_{\R}\p_{x}(G(x-y, t)\eta(x, t))\eta_{*}(y)^{-1}\biggl(\int_{-\infty}^{y}(u_{0}(\xi)-\chi_{*}(\xi))d\xi\biggl) dy\\
=&\int_{\R}\p_{x}(G(x-y, t)\eta(x, t))z_{0}(y)dy, 
\end{split}
\end{align} 
where $z_{0}(y)$ is defined by \eqref{1-17}. Now, from \eqref{2-8}, \eqref{1-2} and \eqref{1-5}, we can check the following estimate in the same way to get \eqref{3-8}: 
\begin{equation*}
|z_{0}(y)|\le C(1+|y|)^{-(\alpha-1)}, \ y\in \R.
\end{equation*}
Thus, we obtain the boundedness of $(1+|y|)^{\alpha-1}z_{0}(y)$. Moreover, from the assumption on $z_{0}(y)$, for any $\e>0$ there is a constant $R=R(\e)>0$ such that
 \begin{align*}
&|z_{0}(y)-c_{\alpha}^{+}(1+|y|)^{-(\alpha-1)}|\le \e(1+|y|)^{-(\alpha-1)}, \ y\ge R, \\
&|z_{0}(y)-c_{\alpha}^{-}(1+|y|)^{-(\alpha-1)}|\le \e(1+|y|)^{-(\alpha-1)}, \ y\le -R. 
 \end{align*}
 From \eqref{1-20} and \eqref{4-6}, we have the following estimate
 \begin{align}\label{4-7}
 \begin{split}
&|U[\psi_{0}](x, t, 0)-Z(x, t)|\\
 &\le \int_{\R}|\p_{x}(G(x-y, t)\eta(x, t))||z_{0}(y)-c_{\alpha}(y)(1+|y|)^{-(\alpha-1)}|dy\\
 &\le \int_{|y|\le R}|\p_{x}(G(x-y, t)\eta(x, t))||z_{0}(y)-c_{\alpha}(y)(1+|y|)^{-(\alpha-1)}|dy\\
 &\ \ \ \ +\e \int_{|y|\ge R}|\p_{x}(G(x-y, t)\eta(x, t))|(1+|y|)^{-(\alpha-1)}dy \\
 &\le C\sum_{n=0}^{1}\|\p_{x}^{1-n}\eta(\cdot, t)\|_{L^{\infty}}\|\p_{x}^{n}G(\cdot, t)\|_{L^{\infty}}\int_{|y|\le R}|z_{0}(y)-c_{\alpha}(y)(1+|y|)^{-(\alpha-1)}|dy\\
 &\ \ \ +\e C\sum_{n=0}^{1}\|\p_{x}^{1-n}\eta(\cdot, t)\|_{L^{\infty}}\int_{\R}|\p_{x}^{n}G(x-y, t)|(1+|y|)^{-(\alpha-1)}dy.\\
 \end{split}
 \end{align}
 Therefore, by using \eqref{4-7}, Lemma \ref{lem2-3}, Lemma \ref{lem2-2} and \eqref{3-11}, we get 
  \begin{align*}
 \begin{split}
 \|U[\psi_{0}](\cdot, t, 0)-Z(\cdot, t)\|_{L^{\infty}}&\le C(1+t)^{-1}\\
&\ \ \ \ +\e C\begin{cases}
(1+t)^{-\alpha/2}, &t\ge1, \ 1<\alpha<2, \\
(1+t)^{-1}\log(2+t), &t\ge1, \ \alpha=2.\end{cases}
  \end{split}
 \end{align*}
 Thus, we obtain 
 \begin{align*}
&\limsup_{t\to \infty}(1+t)^{\alpha/2}\|U[\psi_{0}](\cdot, t, 0)-Z(\cdot, t)\|_{L^{\infty}}\le\e C, \ 1<\alpha<2, \\
&\limsup_{t\to \infty}\frac{(1+t)}{\log(1+t)}\|U[\psi_{0}](\cdot, t, 0)-Z(\cdot, t)\|_{L^{\infty}} \le\e C, \ \alpha=2.  
\end{align*}
Therefore, we get \eqref{4-4} and \eqref{4-5}, because $\e>0$ can be chosen arbitrarily small. 
\end{proof}

\begin{proof}[\rm{\bf{End of the proof of Theorem \ref{main2} for $1<\alpha<2$}}]
First, we shall prove \eqref{1-18}. By using \eqref{1-20} and \eqref{4-3}, we have 
\begin{align}\label{4-8}
\begin{split}
&u(x, t)-\chi(x, t)-Z(x, t)\\
&=U[\psi_{0}](x, t, 0)-Z(x, t)-\frac{b}{2}\int_{0}^{t}U[\p_{x}\psi^{2}(\tau)](x, t, \tau)d\tau\\
&\ \ \ \ -\frac{c}{3}\int_{0}^{t}U[\p_{x}u^{3}(\tau)](x, t, \tau)d\tau-k\int_{0}^{t}U[\p_{x}^{3}u(\tau)](x, t, \tau)d\tau\\
&\equiv U[\psi_{0}](x, t, 0)-Z(x, t)+J_{1}+J_{2}+J_{3}.
\end{split}
\end{align}

\newpage
We shall evaluate $J_{1}$, $J_{2}$ and $J_{3}$. Using Lemma \ref{lem2-5}, Proposition \ref{prop3-2} and Theorem \ref{main1}, we obtain 
\begin{align}\label{4-9}
\begin{split}
&\|J_{1}(\cdot, t)\|_{L^{\infty}} \\
&\le C\sum_{n=0}^{1}(1+t)^{-1/2+n/2}\\
&\ \ \ \ \times \biggl(\int_{0}^{t/2}(t-\tau)^{-1/2-n/2}\| \psi^{2}(\cdot, \tau)\|_{L^{1}}d\tau+\int_{t/2}^{t}(t-\tau)^{-n/2}\| \psi^{2}(\cdot, \tau)\|_{L^{\infty}}d\tau \biggl) \\
&\le C\sum_{n=0}^{1}(1+t)^{-1/2+n/2}\\
&\ \ \ \ \times \biggl(\int_{0}^{t/2}(t-\tau)^{-1/2-n/2}(1+\tau)^{-\alpha+1/2}d\tau+\int_{t/2}^{t}(t-\tau)^{-n/2}(1+\tau)^{-\alpha}d\tau \biggl) \\
&\le C
\begin{cases}
(1+t)^{-1}, &t\ge1, \ 3/2<\alpha<2, \\
(1+t)^{-1}\log(1+t), &t\ge1, \ \alpha=3/2, \\
(1+t)^{-\alpha+1/2}, &t\ge1, \ 1<\alpha<3/2.
\end{cases}
\end{split}
\end{align}

Next we evaluate $J_{2}$. In the same way to get \eqref{4-9}, from Lemma \ref{lem2-5} and Proposition \ref{GE}, we get 
\begin{align}\label{4-10}
\begin{split}
&\|J_{2}(\cdot, t)\|_{L^{\infty}}\\
&\le C\sum_{n=0}^{1}(1+t)^{-1/2+n/2}\\
&\ \ \ \ \times \biggl(\int_{0}^{t/2}(t-\tau)^{-1/2-n/2}\| u^{3}(\cdot, \tau)\|_{L^{1}}d\tau+\int_{t/2}^{t}(t-\tau)^{-n/2}\| u^{3}(\cdot, \tau)\|_{L^{\infty}}d\tau \biggl) \\
&\le C\sum_{n=0}^{1}(1+t)^{-1/2+n/2}\\
&\ \ \ \ \times \biggl(\int_{0}^{t/2}(t-\tau)^{-1/2-n/2}(1+\tau)^{-1}d\tau+\int_{t/2}^{t}(t-\tau)^{-n/2}(1+\tau)^{-3/2}d\tau \biggl) \\
&\le C(1+t)^{-1}\log(1+t), \ t\ge1. 
\end{split}
\end{align}

Finally, we evaluate $I_{3}$. We define $J_{3.1}$ and $J_{3.2}$ are as follows. 
\begin{align}\label{4-11}
\begin{split}
J_{3}=&-k\int_{0}^{t}\int_{\R}\p_{x}(G(x-y, t-\tau)\eta(x, t))(\eta(y, \tau))^{-1}\p_{y}^{2}u(y, \tau)dyd\tau\\
=&-k\sum_{n=0}^{1}\p_{x}^{1-n}\eta(x, t) \\
& \times \left(\int_{0}^{t/2}+\int_{t/2}^{t}\right)\int_{\R}\p_{x}^{n}G(x-y, t-\tau)(\eta(y, \tau))^{-1}\p_{y}^{2}u(y, \tau)dyd\tau\\
\equiv&\ J_{3.1}+J_{3.2}.
\end{split}
\end{align}
Then, by making the integration by parts for $J_{3.1}$, we obtain
\begin{align*}
J_{3.1}=&-k\sum_{n=0}^{1}\p_{x}^{1-n}\eta(x, t)\int_{0}^{t/2}\int_{\R}\p_{x}^{n+1}G(x-y, t-\tau)(\eta(y, \tau))^{-1}\p_{y}u(y, \tau)\\
&+\p_{x}^{n}G(x-y, t-\tau)\p_{y}(\eta(y, \tau))^{-1}\p_{y}u(y, \tau)dyd\tau.
\end{align*}
Using Young's inequality, Lemma \ref{lem2-3}, Lemma \ref{lem2-2}, Proposition \ref{GE} and \eqref{2-8}, we have 
\begin{align}\label{4-12}
\begin{split}
&\|J_{3.1}(\cdot, t)\|_{L^{\infty}}\\
&\le C\sum_{n=0}^{1}\|\p_{x}^{1-n}\eta(\cdot, t)\|_{L^{\infty}}\int_{0}^{t/2}(\|\p_{x}^{n+1}G(t-\tau)*(\eta^{-1}u_{x})(\tau)\|_{L^{\infty}}\\
&\ \ \ \ +\|\p_{x}^{n}G(t-\tau)*((\eta^{-1})_{x}u_{x})(\tau)\|_{L^{\infty}})d\tau \\
&\le C\sum_{n=0}^{1}(1+t)^{-1/2+n/2}\int_{0}^{t/2}(\|\p_{x}^{n+1}G(\cdot, t-\tau)\|_{L^{\infty}}\|u_{x}(\cdot, \tau)\|_{L^{1}}\\
&\ \ \ \ +\|\p_{x}^{n}G(\cdot, t-\tau)\|_{L^{\infty}}\|u_{x}(\cdot, \tau)\|_{L^{2}}\|\p_{x}(\eta(\cdot, \tau)^{-1})\|_{L^{2}})d\tau  \\
&\le C\sum_{n=0}^{1}(1+t)^{-1/2+n/2}\\
&\ \ \ \ \times \int_{0}^{t/2}\left((t-\tau)^{-1-n/2}\tau^{-1/2}(1+\tau^{-1/4})+(t-\tau)^{-1/2-n/2}(1+\tau)^{-1}\right)d\tau\\
&\le C(1+t)^{-1}+C(1+t)^{-1}\log(1+t)\\
&\le C(1+t)^{-1}\log(1+t), \ t\ge1.
\end{split}
\end{align}
On the other hand, for $J_{3.2}$, the following estimate is obtained: 
\begin{align}\label{4-13}
\begin{split}
\|J_{3.2}(\cdot, t)\|_{L^{\infty}}\le&\ C\sum_{n=0}^{1}\|\p_{x}^{1-n}\eta(\cdot, t)\|_{L^{\infty}}\int_{t/2}^{t}\|\p_{x}^{n}G(\cdot, t-\tau)\|_{L^{1}}\|u_{xx}(\cdot, \tau)\|_{L^{\infty}}d\tau \\
\le&\ C\sum_{n=0}^{1}(1+t)^{-1/2+n/2}\int_{t/2}^{t}(t-\tau)^{-n/2}(1+\tau)^{-3/2}d\tau\\
\le&\ C(1+t)^{-1}, \ t\ge1.
\end{split}
\end{align}
Therefore, from \eqref{4-11} through \eqref{4-13}, we get 
\begin{equation}\label{4-14}
\|J_{3}(\cdot, t)\|_{L^{\infty}}\le C(1+t)^{-1}\log(1+t), \ t\ge1.
\end{equation}

By using \eqref{4-8}, \eqref{4-9}, \eqref{4-10} and \eqref{4-14}, we have
\begin{align*}
&\|u(\cdot, t)-\chi(\cdot, t)-Z(\cdot, t)\|_{L^{\infty}}\\
&\le \|U[\psi_{0}](\cdot, t, 0)-Z(\cdot, t)\|_{L^{\infty}}+C(1+t)^{-1}\log(1+t)\\
&\ \ \ \ +C\begin{cases}
(1+t)^{-1}, &t\ge1, \ \ 3/2<\alpha<2,  \\
(1+t)^{-1}\log(1+t), &t\ge1, \ \ \alpha=3/2,  \\
(1+t)^{-\alpha+1/2}, &t\ge1, \ \ 1<\alpha<3/2.
\end{cases}
\end{align*}
Therefore, from \eqref{4-4}, we obtain 
\begin{equation*}
\limsup_{t\to \infty}(1+t)^{\alpha/2}\|u(\cdot, t)-\chi(\cdot, t)-Z(\cdot, t)\|_{L^{\infty}}=0.
\end{equation*}
This completes the proof of \eqref{1-18}. 

\newpage
Next, we shall prove \eqref{1-22}. First, we take $x=0$ in \eqref{1-20}, then we have from \eqref{1-4} and \eqref{1-21} 
\begin{align}\label{4-15}
\begin{split}
&Z(0, t)\\
&=\int_{\R}c_{\alpha}(y)\left((\p_{y}G)(-y, t)\eta(0, t)+G(-y, t)\frac{b}{2}\chi(0, t)\eta(0, t)\right)(1+|y|)^{-(\alpha-1)}dy\\
&=c_{\alpha}^{+}\eta_{*}(0)\int_{0}^{\infty}(\p_{y}G)(-y, t)(1+y)^{-(\alpha-1)}dy\\
&\ \ \ +c_{\alpha}^{-}\eta_{*}(0)\int_{-\infty}^{0}(\p_{y}G)(-y, t)(1-y)^{-(\alpha-1)}dy\\
&\ \ \ +\frac{bc_{\alpha}^{+}}{2}\eta_{*}(0)\chi_{*}(0)(1+t)^{-1/2}\int_{0}^{\infty}G(y, t)(1+y)^{-(\alpha-1)}dy \\
&\ \ \ +\frac{bc_{\alpha}^{-}}{2}\eta_{*}(0)\chi_{*}(0)(1+t)^{-1/2}\int_{-\infty}^{0}G(y, t)(1-y)^{-(\alpha-1)}dy\\
&=(c_{\alpha}^{+}-c_{\alpha}^{-})\eta_{*}(0)\int_{0}^{\infty}(\p_{y}G)(-y, t)(1+y)^{-(\alpha-1)}dy\\
&\ \ \ +\frac{b(c_{\alpha}^{+}+c_{\alpha}^{-})}{2}\eta_{*}(0)\chi_{*}(0)(1+t)^{-1/2}\int_{0}^{\infty}G(y, t)(1+y)^{-(\alpha-1)}dy \\
&\equiv L_{1}(t)+L_{2}(t), 
\end{split}
\end{align}
since $\p_{y}G(y, t)$ and $G(y, t)$ are the odd and even functions for $y$-variable, respectively. 

From the mean value theorem, there exists $\theta_{j} \in (0, 1)$ such that 
\begin{equation}\label{4-16}
(1+y)^{-(\alpha-j)}-y^{-(\alpha-j)}=-(\alpha-j)(y+\theta_{j})^{-(\alpha-j+1)}.
\end{equation}
Therefore, we obtain from \eqref{1-21}
\begin{align}\label{4-17}
\begin{split}
L_{1}(t)&=\frac{(c_{\alpha}^{+}-c_{\alpha}^{-})\eta_{*}(0)}{4\sqrt{\pi}}t^{-3/2}\int_{0}^{\infty}e^{-y^{2}/4t}y(1+y)^{-(\alpha-1)}dy\\
&=\frac{(c_{\alpha}^{+}-c_{\alpha}^{-})\eta_{*}(0)}{4\sqrt{\pi}}t^{-3/2}\\
&\ \ \ \ \times \biggl(\int_{0}^{\infty}e^{-y^{2}/4t}y^{2-\alpha}dy+\int_{0}^{\infty}e^{-y^{2}/4t}y((1+y)^{-(\alpha-1)}-y^{-(\alpha-1)})dy\biggl)\\
&=\frac{(c_{\alpha}^{+}-c_{\alpha}^{-})\eta_{*}(0)}{4\sqrt{\pi}}t^{-3/2}\\
&\ \ \ \ \times \biggl(2^{2-\alpha}t^{(3-\alpha)/2}\Gamma\biggl(\frac{3-\alpha}{2}\biggl)-(\alpha-1)\int_{0}^{\infty}e^{-y^{2}/4t}y(y+\theta_{1})^{-\alpha}dy\biggl), 
\end{split}
\end{align}
since
\begin{equation}\label{4-18}
\int_{0}^{\infty}e^{-y^{2}/4t}y^{j-\alpha}dy=2^{j-\alpha}t^{(j+1-\alpha)/2}\Gamma \biggl(\frac{j+1-\alpha}{2}\biggl), \ \ j\ge1,   
\end{equation}
where $\Gamma(s)$ is the Gamma function for $s>0$. Similarly, by making the integration by parts, we have from \eqref{4-16} and \eqref{4-18}
\begin{align}\label{4-19}
\begin{split}
L_{2}(t)&=\frac{b(c_{\alpha}^{+}+c_{\alpha}^{-})}{4\sqrt{\pi}}\eta_{*}(0)\chi_{*}(0)(1+t)^{-1/2}t^{-1/2}\int_{0}^{\infty}e^{-y^{2}/4t}(1+y)^{-(\alpha-1)}dy \\
&=\frac{b(c_{\alpha}^{+}+c_{\alpha}^{-})}{4\sqrt{\pi}}\eta_{*}(0)\chi_{*}(0)(1+t)^{-1/2}t^{-1/2}\\
&\ \ \ \ \times \biggl(\frac{1}{\alpha-2}+\frac{1}{2(2-\alpha)}t^{-1}\int_{0}^{\infty}e^{-y^{2}/4t}y(1+y)^{-(\alpha-2)}dy\biggl) \\
&=\frac{b(c_{\alpha}^{+}+c_{\alpha}^{-})}{4\sqrt{\pi}}\eta_{*}(0)\chi_{*}(0)(1+t)^{-1/2}t^{-1/2}\biggl(\frac{1}{\alpha-2}+\frac{1}{2(2-\alpha)}t^{-1}\\
&\ \ \ \ \times \biggl(2^{3-\alpha}t^{2-(\alpha/2)}\Gamma \biggl(2-\frac{\alpha}{2}\biggl)+(2-\alpha)\int_{0}^{\infty}e^{-y^{2}/4t}y(y+\theta_{2})^{-(\alpha-1)}dy\biggl)\biggl). \\
\end{split}
\end{align}
Therefore, from \eqref{4-15}, \eqref{4-17} and \eqref{4-19}, we obtain 
\begin{align}\label{4-20}
\begin{split}
&\|Z(\cdot, t)\|_{L^{\infty}}\\
&\ge|Z(0, t)|\\
&=|L_{1}(t)+L_{2}(t)| \\
&=\frac{|\eta_{*}(0)|}{4\sqrt{\pi}}
\biggl| 2^{2-\alpha}(c_{\alpha}^{+}-c_{\alpha}^{-})\Gamma\biggl(\frac{3-\alpha}{2}\biggl)t^{-\alpha/2}\\
&\ \ \ \ -(c_{\alpha}^{+}-c_{\alpha}^{-})(\alpha-1)t^{-3/2}\int_{0}^{\infty}e^{-y^{2}/4t}y(y+\theta_{1})^{-\alpha}dy \\
&\ \ \ \ -\frac{(c_{\alpha}^{+}+c_{\alpha}^{-})b\chi_{*}(0)}{2-\alpha}(1+t)^{-1/2}t^{-1/2}\\
&\ \ \ \ +\frac{2^{2-\alpha}(c_{\alpha}^{+}+c_{\alpha}^{-})b\chi_{*}(0)}{2-\alpha}\Gamma\biggl(2-\frac{\alpha}{2}\biggl)(1+t)^{-1/2}t^{1/2-\alpha/2}\\
&\ \ \ \ +\frac{(c_{\alpha}^{+}+c_{\alpha}^{-})b\chi_{*}(0)}{2}(1+t)^{-1/2}t^{-3/2}\int_{0}^{\infty}e^{-y^{2}/4t}y(y+\theta_{2})^{-(\alpha-1)}dy\biggl| \\
&=\frac{|\eta_{*}(0)|}{4\sqrt{\pi}}
\biggl| 2^{2-\alpha}\biggl((c_{\alpha}^{+}-c_{\alpha}^{-})\Gamma\biggl(\frac{3-\alpha}{2}\biggl)+\frac{(c_{\alpha}^{+}+c_{\alpha}^{-})b\chi_{*}(0)}{2-\alpha}\Gamma\biggl(2-\frac{\alpha}{2}\biggl)\biggl)t^{-\alpha/2}\\
&\ \ \ \ -\frac{(c_{\alpha}^{+}+c_{\alpha}^{-})b\chi_{*}(0)}{2-\alpha}(1+t)^{-1/2}t^{-1/2} \\
&\ \ \ \ -(c_{\alpha}^{+}-c_{\alpha}^{-})(\alpha-1)t^{-3/2}\int_{0}^{\infty}e^{-y^{2}/4t}y(y+\theta_{1})^{-\alpha}dy \\
&\ \ \ \ +\frac{2^{2-\alpha}(c_{\alpha}^{+}+c_{\alpha}^{-})b\chi_{*}(0)}{2-\alpha}\Gamma\biggl(2-\frac{\alpha}{2}\biggl)((1+t)^{-1/2}-t^{-1/2})t^{1/2-\alpha/2}\\
&\ \ \ \ +\frac{(c_{\alpha}^{+}+c_{\alpha}^{-})b\chi_{*}(0)}{2}(1+t)^{-1/2}t^{-3/2}\int_{0}^{\infty}e^{-y^{2}/4t}y(y+\theta_{2})^{-(\alpha-1)}dy\biggl| \\
&\equiv \frac{|\eta_{*}(0)|}{4\sqrt{\pi}}|M_{1}(t)+M_{2}(t)+M_{3}(t)+M_{4}(t)+M_{5}(t)|. 
\end{split}
\end{align}
For $M_{1}(t)$, we get
\begin{equation}\label{4-21}
|M_{1}(t)|\ge 2^{2-\alpha}|\beta_{0}|(1+t)^{-\alpha/2}, 
\end{equation}
where $\beta_{0}$ is defined by \eqref{1-24}. Obviously, we have
\begin{equation}\label{4-22}
|M_{2}(t)|\le \frac{|(c_{\alpha}^{+}+c_{\alpha}^{-})b\chi_{*}(0)|}{2-\alpha}t^{-1}.
\end{equation}
On the other hand, for $M_{3}(t)$, we have from \eqref{4-18}
\begin{align}\label{4-23}
\begin{split}
|M_{3}(t)|&\le (\alpha-1)|c_{\alpha}^{+}-c_{\alpha}^{-}|t^{-3/2}\int_{0}^{\infty}e^{-y^{2}/4t}y^{1-\alpha}dy\\
&=\frac{(\alpha-1)|c_{\alpha}^{+}-c_{\alpha}^{-}|}{2^{\alpha-1}}\Gamma \biggl(1-\frac{\alpha}{2}\biggl)t^{-1/2-\alpha/2}.\\
\end{split}
\end{align}
It is easy to see that 
\begin{align}\label{4-24}
\begin{split}
|M_{4}(t)|&\le \frac{2^{2-\alpha}|(c_{\alpha}^{+}+c_{\alpha}^{-})b\chi_{*}(0)|}{2-\alpha}\Gamma \biggl(2-\frac{\alpha}{2}\biggl)(t^{-1/2}-(1+t)^{-1/2})t^{1/2-\alpha/2}\\
&\le \frac{2^{2-\alpha}|(c_{\alpha}^{+}+c_{\alpha}^{-})b\chi_{*}(0)|}{2-\alpha}\Gamma \biggl(2-\frac{\alpha}{2}\biggl)t^{-1-\alpha/2}.
\end{split}
\end{align}
Finally, for $M_{5}(t)$, we obtain from \eqref{4-18}
\begin{align}\label{4-25}
\begin{split}
|M_{5}(t)|&\le\frac{|(c_{\alpha}^{+}+c_{\alpha}^{-})b\chi_{*}(0)|}{2}(1+t)^{-1/2}t^{-3/2}\int_{0}^{\infty}e^{-y^{2}/4t}y^{-(\alpha-2)}dy\\
&\le \frac{|(c_{\alpha}^{+}+c_{\alpha}^{-})b\chi_{*}(0)|}{2^{\alpha-1}}\Gamma\biggl(\frac{3-\alpha}{2}\biggl)t^{-1/2-\alpha/2}.
\end{split}
\end{align}
Therefore, combining \eqref{4-20} through \eqref{4-25}, we have
\begin{align*}
\|Z(\cdot, t)\|_{L^{\infty}}&\ge \frac{|\eta_{*}(0)|}{4\sqrt{\pi}}(|M_{1}(t)|-|M_{2}(t)|-|M_{3}(t)|-|M_{4}(t)|-|M_{5}(t)|) \\
&\ge  \frac{|\eta_{*}(0)|}{4\sqrt{\pi}}\biggl(2^{2-\alpha}|\beta_{0}|(1+t)^{-\alpha/2}-\frac{|(c_{\alpha}^{+}+c_{\alpha}^{-})b\chi_{*}(0)|}{2-\alpha}t^{-1}\\
&\ \ \ \ -\frac{(\alpha-1)|c_{\alpha}^{+}-c_{\alpha}^{-}|}{2^{\alpha-1}}\Gamma \biggl(1-\frac{\alpha}{2}\biggl)t^{-1/2-\alpha/2} \\
&\ \ \ \ -\frac{2^{2-\alpha}|(c_{\alpha}^{+}+c_{\alpha}^{-})b\chi_{*}(0)|}{2-\alpha}\Gamma \biggl(2-\frac{\alpha}{2}\biggl)t^{-1-\alpha/2}\\
&\ \ \ \ -\frac{|(c_{\alpha}^{+}+c_{\alpha}^{-})b\chi_{*}(0)|}{2^{\alpha-1}}\Gamma\biggl(\frac{3-\alpha}{2}\biggl)t^{-1/2-\alpha/2}\biggl).
\end{align*}
Hence, there is a positive constant $\nu_{0}$ such that \eqref{1-22} holds. This completes the proof of Theorem \ref{main2} for $1<\alpha<2$. 
\end{proof}

\section{Proof of Theorem \ref{main2} for $\alpha=2$}
In this last section, we shall completes the proof of Theorem \ref{main2}. Before proving \eqref{1-19} and \eqref{1-23}, we recall the following fact derived in~\cite{Fukuda}. We consider 
\begin{align}\label{5-1}
\begin{split}
v_{t}+(b\chi v)_{x}-v_{xx}&=-\biggl(\frac{c}{3}\chi^{3}\biggl)_{x}-k\chi_{xxx}, \ \ t>0, \ \ x\in \R, \\
v(x, 0)&=0, \ \ x\in \R.
\end{split}
\end{align}
The leading term of the solution $v(x, t)$ to \eqref{5-1} is given by $V(x, t)$ defined by \eqref{1-11}. Actually, we have the following asymptotic formula (for the proof, see Proposition 4.3 in~\cite{Fukuda}). 
\begin{prop}\label{prop5-1}
Assume that $|\delta| \le 1$. Then the estimate 
\begin{equation}\label{5-2}
\|v(\cdot, t)-V(\cdot, t)\|_{L^{\infty}} \le C|\delta|(1+t)^{-1}, \ t\ge1
\end{equation}
holds. Here $v(x, t)$ is the solution to \eqref{5-1} and $V(x, t)$ is defined by \eqref{1-11}.
\end{prop}

\begin{proof}[\rm{\bf{End of the proof of Theorem \ref{main2} for $\alpha=2$}}]
First, we shall prove \eqref{1-19}. We set 
\begin{align*}
\phi(x, t)&\equiv u(x, t)-\chi(x, t)-v(x, t)\\
&=\psi(x, t)-v(x, t).
\end{align*}
Then, from \eqref{1-1}, \eqref{1-6}, \eqref{4-1} and \eqref{5-1}, $\phi(x, t)$ satisfies the following equation:
\begin{align}\label{5-3}
\begin{split}
\phi_{t}+(b\chi \phi)_{x}-\phi_{xx}&=-\biggl(\frac{b}{2}\psi^{2}\biggl)_{x}-\biggl(\frac{c}{3}\psi^{3}\biggl)_{x}-c(u\chi \psi)_{x}-k\psi_{xxx}, \ t>0, \ x\in \R, \\
\phi(x, 0)&=\psi_{0}(x)=u_{0}(x)-\chi_{*}(x), \ x\in \R.
\end{split}
\end{align}
From Lemma \ref{lem2-4}, we obtain 
\begin{align}\label{5-4}
\begin{split}
\phi(x, t)=&\ U[\psi_{0}](x, t, 0)-\frac{b}{2}\int_{0}^{t}U[\p_{x}\psi^{2}(\tau)](x, t, \tau)d\tau\\
&-\frac{c}{3}\int_{0}^{t}U[\p_{x}\psi^{3}(\tau)](x, t, \tau)d\tau-c\int_{0}^{t}U[\p_{x}(u\chi \psi)(\tau)](x, t, \tau)d\tau\\
&-k\int_{0}^{t}U[\p_{x}^{3}\psi(\tau)](x, t, \tau)d\tau\\
\equiv&\ U[\psi_{0}](x, t, 0)+K_{1}+K_{2}+K_{3}+K_{4}. 
\end{split}
\end{align}
Thus, we have 
\begin{align}\label{5-5}
\begin{split}
&\ u(x, t)-\chi(x, t)-Z(x, t)-V(x, t)\\
=&\ U[\psi_{0}](x, t, 0)-Z(x, t)+v(x, t)-V(x, t)+K_{1}+K_{2}+K_{3}+K_{4}, 
\end{split}
\end{align}
where $Z(x, t)$ and $V(x, t)$ are defined by \eqref{1-20} and \eqref{1-11}, respectively.

We evaluate $K_{1}$, $K_{2}$, $K_{3}$ and $K_{4}$. In the same way to get \eqref{4-9}, from Lemma \ref{lem2-5}, Theorem \ref{main1} and Proposition \ref{prop3-2} for $\alpha=2$, we obtain  
\begin{align}\label{5-6}
\begin{split}
&\|K_{1}(\cdot, t)\|_{L^{\infty}}\\
&\le C\sum_{n=0}^{1}(1+t)^{-1/2+n/2}\\
&\ \ \ \ \times \biggl(\int_{0}^{t/2}(t-\tau)^{-1/2-n/2}\| \psi^{2}(\cdot, \tau)\|_{L^{1}}d\tau+\int_{t/2}^{t}(t-\tau)^{-n/2}\| \psi^{2}(\cdot, \tau)\|_{L^{\infty}}d\tau \biggl) \\
&\le C\sum_{n=0}^{1}(1+t)^{-1/2+n/2}\biggl(\int_{0}^{t/2}(t-\tau)^{-1/2-n/2}(1+\tau)^{-3/2}\log(2+\tau)^{2}d\tau\\
&\ \ \ \ +\int_{t/2}^{t}(t-\tau)^{-n/2}(1+\tau)^{-2}\log(2+\tau)^{2}d\tau \biggl) \\
&\le C(1+t)^{-1}, \ t\ge1. 
\end{split}
\end{align}
We omit the evaluation of $K_{2}$, since we can easily evaluate it in the same way as the above calculation. Actually, it follows that 
\begin{equation}\label{5-7}
\|K_{2}(\cdot, t)\|_{L^{\infty}}\le C(1+t)^{-1}, \ \ge1.
\end{equation}

Next, we evaluate $K_{3}$. Using Lemma \ref{lem2-5}, the Schwarz inequality, \\Proposition \ref{GE}, Lemma \ref{2-2} and Theorem \ref{main1} for $\alpha=2$, we obtain 
\begin{align}\label{5-8}
\begin{split}
&\|K_{3}(\cdot, t)\|_{L^{\infty}}\\
&\le C\sum_{n=0}^{1}(1+t)^{-1/2+n/2}\biggl(\int_{0}^{t/2}(t-\tau)^{-1/2-n/2}\| (u\chi \psi)(\cdot, \tau)\|_{L^{1}}d\tau\\
&\ \ \ \ +\int_{t/2}^{t}(t-\tau)^{-n/2}\| (u\chi \psi)(\cdot, \tau)\|_{L^{\infty}}d\tau \biggl) \\
&\le C\sum_{n=0}^{1}(1+t)^{-1/2+n/2}\biggl(\int_{0}^{t/2}(t-\tau)^{-1/2-n/2}\|u(\cdot, t)\|_{L^{2}}\|\chi(\cdot, t)\|_{L^{2}}\|\psi(\cdot, \tau)\|_{L^{\infty}}d\tau\\
&\ \ \ \ +\int_{t/2}^{t}(t-\tau)^{-n/2}\|u(\cdot, t)\|_{L^{\infty}}\|\chi(\cdot, t)\|_{L^{\infty}}\|\psi(\cdot, \tau)\|_{L^{\infty}}d\tau \biggl) \\
&\le C\sum_{n=0}^{1}(1+t)^{-1/2+n/2}\biggl(\int_{0}^{t/2}(t-\tau)^{-1/2-n/2}(1+\tau)^{-3/2}\log(2+\tau)d\tau\\
&\ \ \ \ +\int_{t/2}^{t}(t-\tau)^{-n/2}(1+\tau)^{-2}\log(2+\tau)d\tau \biggl) \\
&\le C(1+t)^{-1}, \ t\ge1. 
\end{split}
\end{align}

Finally, we evaluate $K_{4}$. From the definition of $K_{4}$, it follows that 
\begin{align*}
\begin{split}
K_{4}=&-k\int_{0}^{t}\int_{\R}\p_{x}(G(x-y, t-\tau)\eta(x, t))(\eta(y, \tau))^{-1}\p_{y}^{2}\psi(y, \tau)dyd\tau\\
=&-k\sum_{n=0}^{1}\p_{x}^{1-n}\eta(x, t)g_{n}(x, t),
\end{split}
\end{align*}
where we put 
\begin{equation*}
g_{n}(x, t)\equiv \int_{0}^{t}\int_{\R}\p_{x}^{n}G(x-y, t-\tau)(\eta(y, \tau))^{-1}\p_{y}^{2}\psi(y, \tau)dyd\tau.
\end{equation*}
Then, we have from Lemma \ref{lem2-3}
\begin{align}\label{5-9}
\begin{split}
\|K_{4}(\cdot, t)\|_{L^{\infty}}&\le C\sum_{n=0}^{1}\|\p_{x}^{1-n}\eta(\cdot, t)\|_{L^{\infty}}\|g_{n}(\cdot, t)\|_{L^{\infty}}\\
&\le C\sum_{n=0}^{1}(1+t)^{-1/2+n/2}\|g_{n}(\cdot, t)\|_{L^{\infty}}.
\end{split}
\end{align}
\newpage 
\noindent 
By making the integration by parts for the first part of the $\tau$-integral of $g_{n}(x, t)$, we have
\begin{align*}
\begin{split}
g_{n}(x, t)=&\ \biggl(\int_{0}^{t/2}+\int_{t/2}^{t}\biggl)\int_{\R}\p_{x}^{n}G(x-y, t-\tau)(\eta(y, \tau))^{-1}\p_{y}^{2}\psi(y, \tau)dyd\tau\\
=&\int_{0}^{t/2}\int_{\R}\p_{x}^{n+2}G(x-y, t-\tau)(\eta(y, \tau))^{-1}\psi(y, \tau)dyd\tau\\
&-2\int_{0}^{t/2}\int_{\R}\p_{x}^{n+1}G(x-y, t-\tau)\p_{y}(\eta(y, \tau))^{-1}\psi(y, \tau)dyd\tau\\
&+\int_{0}^{t}\int_{\R}\p_{x}^{n}G(x-y, t-\tau)\p_{y}^{2}(\eta(y, \tau))^{-1}\psi(y, \tau)dyd\tau \\
&+\int_{t/2}^{t}\int_{\R}\p_{x}^{n}G(x-y, t-\tau)(\eta(y, \tau))^{-1}\p_{y}^{2}\psi(y, \tau)dyd\tau.
\end{split}
\end{align*}
Therefore, from Young's inequality, the Schwarz inequality, Proposition \ref{GE}, Lemma \ref{lem2-2}, Lemma \ref{lem2-3} and Proposition \ref{prop3-1}, we have
\begin{align}\label{5-10}
\begin{split}
\|g_{n}(\cdot, t)\|_{L^{\infty}}\le&\ C\int_{0}^{t/2}\|\p_{x}^{n+2}G(\cdot, t-\tau)\|_{L^{\infty}}(\|u(\cdot, \tau)\|_{L^{1}}+\|\chi(\cdot, t)\|_{L^{1}})d\tau \\
&+C\int_{0}^{t/2}\|\p_{x}^{n+1}G(\cdot, t-\tau)\|_{L^{\infty}}\|\p_{x}(\eta(\cdot, \tau)^{-1})\|_{L^{2}}\|\psi(\cdot, \tau)\|_{L^{2}}d\tau\\
&+C\int_{0}^{t/2}\|\p_{x}^{n}G(\cdot, t-\tau)\|_{L^{\infty}}\|\p_{x}^{2}(\eta(\cdot, \tau)^{-1})\|_{L^{2}}\|\psi(\cdot, \tau)\|_{L^{2}}d\tau\\
&+C\int_{t/2}^{t}\|\p_{x}^{n}G(\cdot, t-\tau)\|_{L^{1}}(\|\p_{x}^{2}u(\cdot, \tau)\|_{L^{\infty}}+\|\p_{x}^{2}\chi(\cdot, \tau)\|_{L^{\infty}})d\tau\\
\le&\ C\int_{0}^{t/2}(t-\tau)^{-3/2-n/2}(1+\tau^{-1/4})d\tau\\
&+C\int_{0}^{t/2}(t-\tau)^{-1-n/2}(1+\tau)^{-1}\log(2+\tau)d\tau\\
&+C\int_{0}^{t/2}(t-\tau)^{-1/2-n/2}(1+\tau)^{-3/2}\log(2+\tau)d\tau\\
&+C\int_{t/2}^{t}(t-\tau)^{-n/2}(1+\tau)^{-3/2}d\tau\\
\le&\ C(1+t)^{-1/2-n/2}, \ t\ge1.
\end{split}
\end{align}
Combining \eqref{5-9} and \eqref{5-10}, we get 
\begin{equation}\label{5-11}
\|K_{4}(\cdot, t)\|_{L^{\infty}}\le C(1+t)^{-1}, \ \ge1.
\end{equation}

Summing up \eqref{5-2}, \eqref{5-6}, \eqref{5-7}, \eqref{5-8} and \eqref{5-11}, we have
\begin{align*}
&\|u(\cdot, t)-\chi(\cdot, t)-Z(\cdot, t)-V(\cdot, t)\|_{L^{\infty}}\\
&\le \|U[\psi_{0}](\cdot, t, 0)-Z(\cdot, t)\|_{L^{\infty}}+C(1+t)^{-1}, \ \ t\ge1.
\end{align*}
Therefore, from \eqref{4-5}, we obtain 
\begin{equation*}
\limsup_{t\to \infty}\frac{(1+t)}{\log(1+t)}\|u(\cdot, t)-\chi(\cdot, t)-Z(\cdot, t)-V(\cdot, t)\|_{L^{\infty}}=0.
\end{equation*}
Thus, we completes the proof of \eqref{1-19}. 
\newpage

Finally, we shall derive the lower bound estimate of $Z(x, t)+V(x, t)$, that is \eqref{1-23}. First, we take $x=0$ in \eqref{1-11}. Then, from \eqref{1-12}, it follows that 
\begin{equation}\label{5-12}
V(0, t)=-\frac{bd}{4\sqrt{\pi}}\biggl(\frac{b^{2}k}{8}+\frac{c}{3}\biggl)\chi_{*}(0)\eta_{*}(0)(1+t)^{-1}\log(1+t).
\end{equation}
Combining \eqref{4-15} and \eqref{5-12}, we have from the triangle inequality and \eqref{1-21}
\begin{align}\label{5-13}
\begin{split}
&\|Z(\cdot, t)+V(\cdot, t)\|_{L^{\infty}}\\
&\ge |Z(0, t)+V(0, t)|\\
&=\biggl|\frac{(c_{2}^{+}-c_{2}^{-})\eta_{*}(0)}{4\sqrt{\pi}}t^{-3/2}\int_{0}^{\infty}e^{-y^{2}/4t}y(1+y)^{-1}dy\\
&\ \ \ \ +\frac{(c_{2}^{+}+c_{2}^{-})b\chi_{*}(0)\eta_{*}(0)}{4\sqrt{\pi}}(1+t)^{-1/2}t^{-1/2}\int_{0}^{\infty}e^{-y^{2}/4t}(1+y)^{-1}dy\\
&\ \ \ \ -\frac{bd}{4\sqrt{\pi}}\biggl(\frac{b^{2}k}{8}+\frac{c}{3}\biggl)\chi_{*}(0)\eta_{*}(0)(1+t)^{-1}\log(1+t)\biggl|\\
&\ge \frac{|b|}{4\sqrt{\pi}}|\chi_{*}(0)||\eta_{*}(0)|(1+t)^{-1/2}\bigg|(c_{2}^{+}+c_{2}^{-})t^{-1/2}\int_{0}^{\infty}e^{-y^{2}/4t}(1+y)^{-1}dy\\
&\ \ \ \ -d\biggl(\frac{b^{2}k}{8}+\frac{c}{3}\biggl)(1+t)^{-1/2}\log(1+t)\biggl|-\frac{|(c_{2}^{+}-c_{2}^{-})\eta_{*}(0)|}{4\sqrt{\pi}}t^{-3/2}\int_{0}^{\infty}e^{-y^{2}/4t}dy\\
&\equiv \frac{|b|}{4\sqrt{\pi}}|\chi_{*}(0)||\eta_{*}(0)|(1+t)^{-1/2}W(t)-\frac{|(c_{2}^{+}-c_{2}^{-})\eta_{*}(0)|}{4}t^{-1}.\\
\end{split}
\end{align}
Now, we evaluate $W(t)$ from below. Splitting the $y$-integral and using the triangle inequality, we obtain 
\begin{align}\label{5-14}
\begin{split}
W(t)=&\ \biggl|(c_{2}^{+}+c_{2}^{-})t^{-1/2}\biggl(\int_{0}^{\sqrt{1+t}-1}+\int_{\sqrt{1+t}-1}^{\infty}\biggl)e^{-y^{2}/4t}(1+y)^{-1}dy\\
&\ -d\biggl(\frac{b^{2}k}{8}+\frac{c}{3}\biggl)(1+t)^{-1/2}\log(1+t)\biggl|\\
\ge&\ \biggl|(c_{2}^{+}+c_{2}^{-})t^{-1/2}\int_{0}^{\sqrt{1+t}-1}e^{-y^{2}/4t}(1+y)^{-1}dy\\
&\ -d\biggl(\frac{b^{2}k}{8}+\frac{c}{3}\biggl)(1+t)^{-1/2}\log(1+t)\biggl|\\
&\ -\biggl|(c_{2}^{+}+c_{2}^{-})t^{-1/2}\int_{\sqrt{1+t}-1}^{\infty}e^{-y^{2}/4t}(1+y)^{-1}dy\biggl|\\
\equiv&\ W_{1}(t)-W_{2}(t).
\end{split}
\end{align}
From the mean value theorem, there exists $\theta \in(0, 1)$ such that 
\begin{equation*}
e^{-y^{2}/4t}=1-\frac{y^{2}}{4t}e^{-\theta y^{2}/4t}. 
\end{equation*}
Therefore, we obtain 
\begin{align}\label{5-15}
\begin{split}
W_{1}(t)=&\ \biggl|(c_{2}^{+}+c_{2}^{-})t^{-1/2}\int_{0}^{\sqrt{1+t}-1}\biggl(1-\frac{y^{2}}{4t}e^{-\theta y^{2}/4t}\biggl)(1+y)^{-1}dy\\
&\ -d\biggl(\frac{b^{2}k}{8}+\frac{c}{3}\biggl)(1+t)^{-1/2}\log(1+t)\biggl|\\
\ge&\ \biggl|(c_{2}^{+}+c_{2}^{-})t^{-1/2}\int_{0}^{\sqrt{1+t}-1}(1+y)^{-1}dy\\
&\ -d\biggl(\frac{b^{2}k}{8}+\frac{c}{3}\biggl)(1+t)^{-1/2}\log(1+t)\biggl|\\
&\ -\biggl|\frac{c_{2}^{+}+c_{2}^{-}}{4}t^{-3/2}\int_{0}^{\sqrt{1+t}-1}e^{-\theta y^{2}/4t}y^{2}(1+y)^{-1}dy\biggl|\\
\equiv&\ W_{1.1}(t)-W_{1.2}(t).
\end{split}
\end{align}
For $W_{1.1}(t)$, we have
\begin{align}\label{5-16}
\begin{split}
W_{1.1}(t)=&\ \biggl|\frac{c_{2}^{+}+c_{2}^{-}}{2}t^{-1/2}\log(1+t)-d\biggl(\frac{b^{2}k}{8}+\frac{c}{3}\biggl)(1+t)^{-1/2}\log(1+t)\biggl|\\
=&\ \biggl|\frac{c_{2}^{+}+c_{2}^{-}}{2}t^{-1/2}\log(1+t)-\frac{c_{2}^{+}+c_{2}^{-}}{2}(1+t)^{-1/2}\log(1+t)\\
&\ +\biggl(\frac{c_{2}^{+}+c_{2}^{-}}{2}-d\biggl(\frac{b^{2}k}{8}+\frac{c}{3}\biggl)\biggl)(1+t)^{-1/2}\log(1+t)\biggl| \\
\ge&\ |\beta_{1}|(1+t)^{-1/2}\log(1+t)-\frac{|c_{2}^{+}+c_{2}^{-}|}{2}\log(1+t)(t^{-1/2}-(1+t)^{-1/2})\\
\ge&\ |\beta_{1}|(1+t)^{-1/2}\log(1+t)-\frac{|c_{2}^{+}+c_{2}^{-}|}{2}t^{-3/2}\log(1+t),
\end{split}
\end{align}
where $\beta_{1}$ is defined by \eqref{1-24}. On the other hand, for $W_{1.2}(t)$, we obtain 
\begin{align}\label{5-17}
\begin{split}
W_{1.2}(t)&\le \frac{|c_{2}^{+}+c_{2}^{-}|}{4}t^{-3/2}\int_{0}^{\sqrt{1+t}-1}y\ dy\\
&=\frac{|c_{2}^{+}+c_{2}^{-}|}{8}t^{-3/2}(\sqrt{1+t}-1)^{2}\le \frac{|c_{2}^{+}+c_{2}^{-}|}{8}t^{-1/2}. 
\end{split}
\end{align}
Analogously, for $W_{2}(t)$, it follows that 
\begin{align}\label{5-18}
\begin{split}
W_{2}(t)&\le|c_{2}^{+}+c_{2}^{-}|t^{-1/2}\biggl(\sup_{y\ge \sqrt{1+t}-1}(1+|y|)^{-1}\biggl)\int_{y\ge \sqrt{1+t}-1}e^{-y^{2}/4t}dy\\
&\le \sqrt{\pi}|c_{2}^{+}+c_{2}^{-}|(1+t)^{-1/2}.
\end{split}
\end{align}
Therefore, summing up \eqref{5-13} through \eqref{5-18}, we get 
\begin{align*}
&\|Z(\cdot, t)+V(\cdot, t)\|_{L^{\infty}}\\
&\ge \frac{|b|}{4\sqrt{\pi}}|\chi_{*}(0)||\eta_{*}(0)|\biggl(|\beta_{1}|(1+t)^{-1}\log(1+t)-\frac{|c_{2}^{+}+c_{2}^{-}|}{2}t^{-2}\log(1+t)\\
&\ \ \ \ -|c_{2}^{+}+c_{2}^{-}|\biggl(\sqrt{\pi}+\frac{1}{8}\biggl)t^{-1}\biggl)-\frac{|(c_{2}^{+}-c_{2}^{-})\eta_{*}(0)|}{4}t^{-1}.
\end{align*}
Hence, there is a positive constant $\nu_{1}$ such that \eqref{1-23} holds. This completes the proof of Theorem \ref{main2} for $\alpha=2$. 
\end{proof}

\vskip10pt
\par\noindent
\textbf{Acknowledgments} 

The author would like to express his sincere gratitude to Professor Hideo Kubo for his feedback and valuable advices. 

This study is partially supported by Grant-in-Aid for JSPS Research Fellow \\No.18J1234008 and MEXT through Program for Leading Graduate Schools (Hokkaido University ``Ambitious Leader's Program").


\end{document}